\documentclass[11pt]{amsart}
\usepackage[top=3cm, bottom=3cm, left=2.5cm, right=2.5cm]{geometry}
\usepackage{graphicx} % Required for inserting images
\usepackage{bm,tikz}
\usepackage{mathtools,enumitem}
\usepackage{amsfonts,amsthm,amsmath,amssymb,units}
\usepackage{hyperref} % must be loaded before cleveref
\usepackage{comment}

\renewcommand{\k}{\kappa}
\renewcommand{\t}{\theta}
\newcommand{\R}{\mathbb{R}}
\newcommand{\wQ}{\widehat{Q}}

\newcommand{\wG}{\widehat{G}}
\newcommand{\wF}{\widehat{F}}
\newcommand{\wB}{\widehat{B}}
\newcommand{\wC}{\widehat{C}}
\newcommand{\wJ}{\hat{J}}
\newcommand{\fW}{\mathsf{W}}
\newcommand{\fX}{\mathsf{X}}
\newcommand{\fY}{\mathsf{Y}}
\newcommand{\fZ}{\mathsf{Z}}

\newcommand{\mA}{\mathcal{A}}
\newcommand{\mP}{\mathcal{P}}
\newcommand{\mB}{\mathcal{B}}

\newcommand{\mM}{\mathcal{M}}
\newcommand{\mH}{\mathcal{H}}

\newcommand{\mC}{\mathcal{C}}
\newcommand{\mD}{\mathcal{D}}
\newcommand{\mY}{\mathcal{Y}}
\newcommand{\mZ}{\mathcal{Z}}

\newtheorem{theorem}{Theorem}
\newtheorem{lemma}[theorem]{Lemma}
\newtheorem{corollary}[theorem]{Corollary}
\newtheorem{proposition}[theorem]{Proposition}
\theoremstyle{definition}
\newtheorem{remark}[theorem]{Remark}
\newtheorem{example}[theorem]{Example}

\usepackage[nameinlink]{cleveref}
\usepackage[version=3]{mhchem}

\newcommand\red[1]{\textcolor{red}{#1}}

\newenvironment{ex}
  {\pushQED{\qed}\example}
  {\popQED\endexample}

\title[On fold and cusp bifurcations in mass action networks]{On the computation of fold and cusp bifurcations in mass action networks via dual equations}
\author{Murad Banaji}
\date{January-May 2026}

\begin{document}

\begin{abstract}
This paper presents some results useful in identifying or ruling out fold and cusp bifurcations in mass action networks, and hence in studying multistationarity in these networks. The main result is that the occurrence of fold and cusp bifurcations in mass action networks can be confirmed or ruled out using alternative systems of equations, obtained using principles of Gale duality, rather than the original mass action equations. This can, for certain classes of networks, greatly simplify the calculations involved, avoiding the need for explicit centre manifold reduction or Lyapunov-Schmidt reduction when checking bifurcation, nondegeneracy and transversality conditions. Moreover, the approach leads to immediate corollaries on the non-occurrence of bifurcations in classes of networks with certain combinatorial properties. We present a number of examples illustrating the results and corollaries.

\vspace{0.5cm}
\noindent
{\bf MSC.} 37G10, 92E20, 34C08, 05E40
\end{abstract}

\maketitle

\section{Introduction}

Studying bifurcations in chemical reaction networks helps us make sense of their dynamics, and can give us information on important behaviours such as multistationarity, oscillation and chaos. A small sample of classical and modern work using bifurcation theory to study the dynamics of reaction networks includes \cite{frank-kamenetsky:salnikov:1943, Selkov1968, DiCera1989, Gatermann2005, Conradi2007, abphopf, errami2015, Plesa2015, obatake2019, banajiborosnonlinearity, Vassena2023,banaji:boros:hofbauer:2025, Huang2025}. As network size increases, the computational difficulty of studying bifurcations, and in particular of checking nondegeneracy and transversality conditions, tends to increase rapidly (see, for example, the analysis of focal values in the study of Andronov--Hopf and Bautin bifurcations in \cite{banajiborosnonlinearity,boros_github}). Centre-manifold reduction \cite{carr2012applications} or Lyapunov-Schmidt reduction \cite{golubitsky1985} are generally required; but implementing these, for example using computer algebra systems, can be challenging. When checking bifurcation conditions in high-dimensional systems, various alternatives to explicit centre-manifold reduction are available \cite[Chapter 5]{kuznetsov:2023}, but these too ultimately lead to feasibility problems which grow rapidly more complex as the network size increases. It can be tempting to reduce the complexity of the calculations by focussing on basic bifurcation conditions and ignoring questions of nondegeneracy and transversality, but this is risky as reaction networks can robustly display non-generic or incompletely unfolded bifurcations (see \cite{banaji:boros:hofbauer:2023} and examples in \cite{banajiborosnonlinearity,BBH2024smallbif}).

With these remarks in mind, finding approaches which reduce the computational complexity of checking bifurcation conditions, including nondegeneracy and transversality conditions, becomes important. Here, we focus our attention on {\bf mass action networks}, namely, reaction networks with mass action kinetics. One strategy we have previously developed involves {\em inheritance results} \cite{banaji:boros:hofbauer:2025} which allow us to infer dynamical behaviours, including bifurcations, in reaction networks by studying their subnetworks. While these results can allow us to infer the occurrence of bifurcations in large networks without the need for lengthy and complex calculations, they have some limitations: they do not provide a means to rule out bifurcations; and they cannot give us complete descriptions of bifurcation sets. 

Recent work in \cite{banajifeliu26}, using principles of Gale duality \cite{sottile:book} to recast the defining equations for equilibria in mass action networks, opened up another avenue for studying certain bifurcations. We show in this paper that identifying fold and cusp bifurcations in mass action networks, including checking nondegeneracy and transversality, can be accomplished with the help of {\em alternative equations} as constructed in \cite{banajifeliu26}, see also a related construction in \cite{regensburger:gale}. For certain classes of networks, this approach allows us to obtain bifurcation sets for fold and cusp bifurcations using relatively simple algebra, and without the need for explicit centre-manifold or Lyapunov-Schmidt reduction. Moreover, it gives us immediate corollaries about the non-occurrence of bifurcations in certain classes of networks, which are not easily derived from the original mass action equations. 

The remainder of this paper is laid out as follows. In Section~\ref{secexamples}, we present some motivating examples, demonstrating how the theory to be developed can greatly simplify computations involving fold and cusp bifurcations. In Section~\ref{secprelim} we present preliminaries on bifurcations, and in particular fold and cusp bifurcations, in a general context. In Section~\ref{secMA} we focus on mass action networks, recapping and developing some theory from \cite{BBH2024smallbif} and \cite{banajifeliu26}. In Section~\ref{secmain} we present the main result of this paper, \Cref{foldcuspmain}, and its proof. Finally in Section~\ref{seccorollaries} we present some corollaries and concluding remarks. 

\section{Motivating examples}
\label{secexamples}
To motivate the theory to follow, we start with a few examples. All claims are justified using \Cref{foldcuspmain}, the main result of this paper, with \Cref{prop:coord,prop:elimination,prop:prefactor,prop:power} justifying some of the intermediate calculations. Some notions and terminology used in the examples will be elaborated on later. 

We write $\mathbb{R}^n_+:=\{x \in \mathbb{R}^n\,:\,x_i> 0,\,\,i=1,\ldots,n\}$ for the positive orthant in $\mathbb{R}^n$. Given a system of $n$ equations in $n$ variables, say $f(x) = 0$, we will refer to a solution $x_0$ as {\bf degenerate} if $J^{(f)}(x_0)$, the Jacobian matrix of $f$, evaluated at $x_0$, is singular; and as {\bf simply degenerate} if $J^{(f)}(x_0)$ has a simple zero eigenvalue, and no other eigenvalues on the imaginary axis. 

The first example can easily be analysed using the original mass action equations (with some help from \Cref{prop:elimination} below), but is useful to illustrate the approach here. 

\begin{ex}
\label[example]{Ex1}
Consider the following $2$-species, $4$-reaction network with no linear conservation laws:  
\[
\fX \ce{->[\k_1]} \mathsf{0}\,,\quad \mathsf{0} \ce{->[\k_2]} \fX + \fY,\quad \fX + \fY \ce{->[\k_3]} 3\fX\,,\quad 2\fX \ce{->[\k_4]} 3\fX\,.
\]
This network appeared in Example 2 of \cite{BBH2024smallbif}: despite its simplicitly, it admits a Bogdanov--Takens bifurcation, and hence fold, Andronov--Hopf and homoclinic bifurcations. It gives rise to the following (quadratic) mass action differential equations:
\[
\begin{array}{rcl}
\dot x &=& -\k_1 x +\k_2 +2\k_3xy+\k_4x^2\,,\\
\dot y &=& \k_2 - \k_3 xy\,.
\end{array}
\]
An easy direct calculation confirms that the network has a degenerate equilibrium if and only if $\frac{\k_1^2}{\k_2\k_4} = 12$, and this equilibrium is simply degenerate whenever $\k_3 \neq \frac{2}{3}\k_4$. The alternative equations, whose derivation is presented in detail in \cite{banajifeliu26} and more briefly again in Section~\ref{secMA} below, can be written:
\begin{equation}
\label{alteq1}
f(\alpha, \theta):=(4+\theta)\alpha^2 + (4-\theta)\alpha + 1 = 0, \quad \mbox{where} \quad \alpha \in (0,1) \,\,\mbox{and} \,\,\theta := \frac{\k_1^2}{\k_2\k_4}\,.
\end{equation}
From rapid examination of \eqref{alteq1}, the results of \cite{banajifeliu26} tell us that the mass action system has:
  \begin{itemize}
  \item[(i)] no positive equilibria for $\frac{\k_1^2}{\k_2\k_4} < 12$;
  \item[(ii)] two positive nondegenerate equilibria for $\frac{\k_1^2}{\k_2\k_4} > 12$;
  \item[(iii)] a single positive equilibrium, which is degenerate, when $\frac{\k_1^2}{\k_2\k_4} = 12$. 
\end{itemize}
The results to be developed in this paper additionally tell us that the system has:
  \begin{itemize}
  \item[(iv)] a nondegenerate fold bifurcation if and only if $\frac{\k_1^2}{\k_2\k_4} = 12$ and $\k_3 \neq \frac{2}{3}\k_4$; and this bifurcation is always unfolded by the rate constants $\k$.
\end{itemize}
The claims about nondegeneracy and transversality in (iv) rest on the easy calculations $f_{\alpha\alpha}(1/4,12) \neq 0$ and $f_\theta(1/4,12) \neq 0$. 
\end{ex}

For the next network, which appeared in \cite[Example 4.34]{banajifeliu26}, we derive a parameterisation of the cusp bifurcation set which would be somewhat challenging to derive directly from the mass action equations.  
\begin{ex}
 \label{ex342b}
Consider the following bimolecular 3-species, 4-reaction network with a single linear conservation law:
\[
\fX \ce{->[\k_1]} \fY\,,\quad \fY \ce{->[\k_2]} \fZ,\quad 2\fX \ce{->[\k_3]} 2\fZ\,,\quad \fY+\fZ \ce{->[\k_4]} 2\fX\,.
\]
The network gives rise to the following mass action differential equations
\[
\begin{array}{rcl}
\dot x &=& -\k_1 x - 2\k_3 x^2 + 2\k_4 yz\,,\\
\dot y &=& \k_1 x - \k_2 y - \k_4 yz\,,\\
\dot z &=& \k_2 y + 2\k_3 x^2 - \k_4 yz\,.
\end{array}
\]
As $\dot x + \dot y + \dot z = 0$, given any $K \in \R_+$, each {\em positive stoichiometric class} of the form $S_K:=\{(x,y,z) \in \R^3_+\,:\, x+y+z = K\}$ is locally invariant. Following the proof of Theorem~4.1 in \cite{boros:hofbauer:2022}, or by direct calculation, we note that after multiplication by $\frac{1}{xyz}$, the divergence of the mass action vector field is negative everywhere on $\R^3_+$; consequently, any equilibrium degenerate relative to its positive stoichiometric class, say $S_K$, is simply degenerate relative to $S_K$. In this case, the alternative equations, corresponding to the restriction of the system to $S_K$, consist of a single equation
\begin{equation}
  \label{eqpoly2}
f(\alpha; \t_1, \t_2, \t_3, K):=2\t_1\alpha(1-\alpha)+\t_2\alpha^{2}(1-\alpha) + 4\t_3(2-\alpha) -4K\alpha = 0\,,
\end{equation}
where $\alpha \in (0,1)$, and $\t_1 = \k_1/\k_3,\,\t_2=\k_1^2/(\k_2\k_3),\,\,\t_3=\k_2/\k_4$. By the results here, to find fold and cusp bifurcations, it suffices to examine $f$: we complete the analysis for the harder case of cusp bifurcations. We can easily simplify the equations $f = f_\alpha = f_{\alpha \alpha} = 0$ to get
\[
\t_1 = \frac{(1-3\alpha)\t_2}{2},\,\, \t_3 = \frac{\alpha^3 \t_2}{8},\,\, K = \frac{(2-6\alpha+6\alpha^2-\alpha^3)\t_2}{8}\,, 
\]
with $\{(\alpha, \t_2) \in (0,1/3) \times \R_+\}$ (note that $\t_1 > 0$ implies that $\alpha<1/3$, and that $K>0$ holds for all $(\alpha,\theta_2) \in (0,1) \times \R_+$). Further, we have the easy calculations: $f_{\alpha\alpha\alpha} = -6\t_2 < 0$ for all $\t_2 \in \R_+$; and $f_{\t_1}f_{\alpha\t_2}-f_{\t_2}f_{\alpha \t_1} = 2\alpha^2(1-\alpha)^2>0$ for all $\alpha \in (0,1)$. We may now use the results here to translate this information back into the following conclusions in terms of the original state variables $x,y,z$ and parameters $\k_i$: given $(\alpha, \beta, \gamma) \in (0,1/3) \times \R_+ \times \R_+$, when the rate constants are
\[
(\k_1,\k_2,\k_3,\k_4) = \gamma\left(\frac{2}{(1-3\alpha)\beta},\,\,\frac{1}{\beta}, \,\,\frac{4}{(1-3\alpha)^2\beta^2}, \,\,\frac{8}{\alpha^2\beta^2}\right)
\]
a nondegenerate cusp bifurcation, unfolded by the rate constants, occurs at 
\[
(x,y,z) = \left(\frac{\beta(1-3\alpha)(1-\alpha)}{4},\,\,\frac{\alpha\beta(1-\alpha)}{4},\,\,\frac{\alpha^2\beta(2-\alpha)}{8}\right)\,.
\]
Moreover, this completely characterises cusp bifurcations in this network. 
\end{ex}

The following example appeared as Example~4.36 in \cite{banajifeliu26} and corresponds to a phosphorylation mechanism  studied in \cite{Kothamachu2015multistability}. It is more complicated than the previous examples, involving six reactions, six chemical species, and two independent linear conservation laws. For brevity, we do not present a complete analysis of fold and cusp bifurcations in this network, but illustrate how such an analysis is simplified using the approach in this paper. 
\begin{ex}
\label{ex664}
Consider the following bimolecular 6-species, 6-reaction network:
\[
\fX_1 \ce{->[\k_1]} \fX_2 \ce{->[\k_2]} \fX_3 \ce{->[\k_3]} \fX_4, \quad \fX_3+\fX_5 \ce{->[\k_4]} \fX_1 + \fX_6, \quad \fX_4+\fX_5 \ce{->[\k_5]} \fX_2 + \fX_6, \quad \fX_6 \ce{->[\k_6]} \fX_5\,,
\]
leading to the mass action differential equations
\[
\begin{array}{rcl}
\dot x_1 &=& -\k_1 x_1 + \k_4x_3x_5\,,\\
\dot x_2 &=& \k_1x_1-\k_2x_2+\k_5x_4x_5\,,\\
\dot x_3 &=& \k_2x_2 - \k_3x_3 - \k_4x_3x_5\,,\\
\dot x_4 &=& \k_3x_3 - \k_5x_4x_5\,,\\
\dot x_5 &=& -\k_4x_3x_5-\k_5x_4x_5+\k_6x_6\,,\\
\dot x_6 &=& \k_4x_3x_5 + \k_5x_4x_5 - \k_6x_6\,.
\end{array}
\]
We may check with some direct calculations that any degenerate positive equilibrium is simply degenerate w.r.t. its stoichiometric class: for example, we may check with computer algebra that all monomials in the determinant of the second additive compound \cite{muldowney} of the reduced Jacobian matrix \cite{banajipantea} have the same sign which is sufficient to guarantee the claim. In this case, the alternative equations can be written as a pair of polynomial equations
\begin{equation}
\label{eq664}
\begin{array}{rcl}
 \t_5\alpha + \t_6(1-\alpha)\mu & = & K_1(1-\alpha)\,,\\[5pt]
 \t_1\alpha^2\mu + \t_2\alpha\mu + \t_3\alpha(1-\alpha)\mu + \t_4(1-\alpha)^2\mu & = & K_2\alpha\,,
\end{array}
\end{equation}
where $\alpha \in (0,1)$, $\t = (\t_1, \ldots, \t_6) \in \R^6_+$ are rational combinations of the rate constants, and $K=(K_1,K_2) \in \R^2_+$ is a constant corresponding to a choice of some particular stoichiometric class. By \Cref{foldcuspmain}, in order to find fold and cusp bifurcations in this network it is necessary and sufficient to examine \eqref{eq664}. Eliminating $\mu$ and clearing denominators (justified by \Cref{prop:elimination,prop:prefactor} below) reduces \eqref{eq664} to a single cubic equation of the form:
\[
f(\alpha, K,\t):=p_3(K,\t)\alpha^3 + p_2(K,\t)\alpha^2+p_1(K,\t)\alpha + p_0(K,\t) = 0\,,
\]
where each $p_i$ is a quadratic polynomial in $(K,\t)$. From the point of view of studying fold and cusp bifurcations, $p_1, \ldots, p_4$ are the relevant combinations of the eight parameters $\k_1, \ldots, \k_6, K_1, K_2$. A fold bifurcation occurs whenever $(\alpha_0,K_0,\t_0) \in (0,1) \times \R^2_+ \times \R^6_+$ satisfies 
\[
0=f(\alpha_0,K_0,\t_0) = f_\alpha(\alpha_0, K_0, \t_0)\,,
\]
and is nondegenerate and unfolded by the parameters $\t$ (hence, $\k$) provided $f_{\alpha\alpha}(\alpha_0, K_0, \t_0) \neq 0$ and $f_{\t_i}(\alpha_0, K_0, \t_0) \neq 0$ for some $i$. A cusp bifurcation occurs when $(\alpha_0,K_0,\t_0) \in (0,1) \times \R^2_+ \times \R^6_+$ satisfies $0=f(\alpha_0,K_0,\t_0) = f_\alpha(\alpha_0, K_0, \t_0) = f_{\alpha\alpha}(\alpha_0,K_0,\t_0)$, and is nondegenerate and unfolded by the parameters $\t$ (hence, $\k$) provided $f_{\alpha\alpha\alpha}(\alpha_0, K_0, \t_0) \neq 0$ and $(f_{\alpha\t_i}f_{\t_j} - f_{\alpha\t_j}f_{\t_i})(\alpha_0, K_0, \t_0) \neq 0$ for some $i,j$. We may translate these conditions in a natural way into conditions on the concentrations of species $\fX_i$ and rate constants $\k_i$.
\end{ex}

The previous examples demonstrated how using the alternative equations can simplify the calculations needed to confirm fold and cusp bifurcations. The next two examples demonstrate how the techniques can also help us rule out bifurcations. 

\begin{ex}
\label[example]{Ex363a}
Consider the following $3$-species, $6$-reaction network without linear conservation laws:  
\[
\fX \ce{->[\k_1]} 2\fZ\,,\quad 2\fZ \ce{->[\k_2]} \fX + \fZ,\quad \fX + \fZ \ce{->[\k_3]} \fX\,,\quad \fY+\fZ \ce{->[\k_4]} \fY\,,\quad \fX+\fY \ce{->[\k_5]} \fZ\,,\quad \fZ \ce{->[\k_6]} \fX+\fY\,.
\]
The corresponding mass action system has the form
\[
\begin{array}{rcl}
\dot x &=&\k_1x + \k_2z^2 - \k_5xy+\k_6z\,,\\
\dot y &=& -\k_5xy + \k_6z\,,\\
\dot z &=& 2\k_1x - \k_2z^2-\k_3xz - \k_4yz+\k_5xy-\k_6z\,.
\end{array}
\]
This example will be revisited in \Cref{Ex363b}, where we will see that although there are six rate constants, $\k_1, \ldots, \k_6$, from the point of view of bifurcations affecting the equilibrium set, only one parameter combination, $\frac{\k_2^2\k_1\k_5}{\k_3^2\k_4\k_6}$ matters. Thus the network cannot admit a cusp bifurcation for the simple reason that, effectively, it does not have sufficient parameters to unfold this bifurcation. 
\end{ex}

In the next example, the claim that cusp bifurcations are ruled out arises in a similar, but somewhat more subtle, way.
\begin{ex} Consider the following $4$-species, $5$-reaction mass action network with a single linear conservation law:
\[
\fW \ce{->[\k_1]} \fX + \fY,\quad 2\fX \ce{->[\k_2]} \fW+\fX,\quad \fW+\fY \ce{->[\k_3]} \fW + \fZ,\quad \fZ \ce{->[\k_4]} 2\fZ,\quad \fY + \fZ \ce{->[\k_5]} \mathsf{0}\,.
\]
The mass action equations corresponding to this network take the form
\[
\begin{array}{rcl}
\dot w &=& -\k_1w+\k_2x^2\,,\\
\dot x &=& \k_1w-\k_2 x^2\,,\\
\dot y &=&\k_1 w - \k_3wy - \k_5 yz\,,\\
\dot z &=& \k_3 wy + \k_4 z - \k_5 yz\,.
\end{array}
\]
As $\dot w + \dot x = 0$, given any $K \in \R_+$ we may consider the restriction of this system to the positive stoichiometric class $S_K:=\{(w,x,y,z) \in \R^4_+\,:\,w + x = K\}$. We claim that for no choice of $(\k,K)\in \R^5_+ \times \R_+$ does the mass action system admit a cusp bifurcation on $S_K$. This follows as this network belongs to a class of networks in \Cref{propnnn1}(ii) below for which nondegenerate, fully unfolded cusp bifurcations are impossible, even though three or more positive equilibria are not ruled out and cusp points may occur in this class of networks. Key to this claim is the fact that the source complexes of the network are affinely dependent and, as a consequence, the network is ``$\mP$-toric'', implying that the alternative equations decompose in a particular way. 
\end{ex}

\section{Preliminaries}
\label{secprelim}

\subsection{Notation}

The symbol $\bm{1}$ will denote a column vector of ones whose length is inferred from the context; while $I$ will denote the identity matrix whose dimensions are again inferred from the context. Given any vector $v \in \R^n$, we write $D_v$ to mean the diagonal $n \times n$ matrix with the entries of $v$ on the diagonal. We denote the $i$th standard basis vector in $\R^n$ by $e_i$. Given $A, B \subseteq \R^n$, we write $A+B := \{a+b\,:\, a \in A, b\in B\}$ and $AB :=\{a\circ b\,:\,a \in A, b \in B\}$, where ``$\circ$'' denotes the entrywise product. We will freely apply arithmetic operations and functions to vectors entrywise without remarking on this: for example, given $u, v \in \R^n$, we write $v^2$ to mean $v \circ v$ and, provided no entry of $v$ is zero, $u/v$ to mean the vector whose $i$th entry is $u_i/v_i$.

For the remainder of this section we let $U \subseteq \R^n,\,V \subseteq \R^m$ be open, while $f\colon U \times V \to \R^n$, $(y,\beta) \mapsto f(y,\beta)$ is smooth. We may interpret ``smooth'' to mean $C^\infty$, but in practice we require only that all derivatives we need to compute exist and are continuous. 

\subsection{Derivatives and Taylor coefficients}
We use the same notation for multilinear maps and their matrix/ tensor representations, as this should cause no confusion. We write $\partial_yf$ (or $f_y$ more briefly) and $\partial_\beta f$ (or $f_\beta$ more briefly) to mean the derivatives of $f$ w.r.t. its first and second arguments respectively, and $\partial_{(y, \beta)} f$ for the derivative of $f$ w.r.t. both arguments, namely $f_y \oplus f_\beta$ (all of these are linear maps). We use similar notation for higher order derivatives, e.g., $f_{yy}$, $f_{y\beta}$, $f_{yyy}$, etc., regarded as multilinear functions. We always assume that after all differentiation has been carried out, derivatives are evaluated at some point of interest, say $p_0=(y_0, \beta_0)$, which will be clear from the context. 

We use the notation $J^{(f)}:=f_y$, and also write, following notation in \cite{kuznetsov:2023}, $B^{(f)}(\cdot\,,\cdot):= f_{yy}(\cdot\,,\cdot)$, and $C^{(f)}(\cdot\,, \cdot\,, \cdot):= f_{yyy}(\cdot, \cdot, \cdot)$, for the symmetric, multilinear functions representing, up to a scaling factor, the quadratic and cubic terms in the Taylor expansion of $f$ w.r.t. its first argument $y$, evaluated at $p_0$. We omit the superscript $f$ from $J^{(f)}$, $B^{(f)}$ and $C^{(f)}$ when the function is obvious from the context.

Given $u,v\in \R^n$, observe that
\[
B^{(f)}(u,v) := \left.\frac{\partial}{\partial t}\left[\frac{\partial}{\partial t'} f(y_0 + t'v + tu,\beta_0)\right]\right|_{t=t'=0}\,, \quad \mbox{hence} \quad B^{(f)}(u,u):=\left.\frac{\partial^2}{\partial t^2}f(y_0 + tu, \beta_0)\right|_{t=0}\,.
\]
Further,
\[
C^{(f)}(u,u,u) := \left.\frac{\partial^3}{\partial t^3}f(y_0 + tu, \beta_0)\right|_{t=0}\,.
\]

\subsection{Local bifurcations of equilibria}
We will be interested in when nondegenerate fold and cusp bifurcations can occur in mass action networks, and when these are unfolded by the rate constants of the network. Following remarks in \cite[Chapter 2]{kuznetsov:2023} and the discussion in \cite[Section 2]{banaji:boros:hofbauer:2025}, we make some general comments about local bifurcations of equilibria of the parameterised family of vector fields $f(y, \beta)$. 

For our purposes, such bifurcations will always be defined by equations and inequalities involving functions of the Taylor coefficients of $f$ at some zero of $f$. These conditions can be divided into (i) basic bifurcation conditions; (ii) nondegeneracy conditions; and (iii) transversality conditions. While the bifurcation and nondegeneracy conditions tell us about the nature of the degeneracy, and involve only derivatives w.r.t. the state variables $y$, the transversality conditions also involve derivatives w.r.t. the parameters $\beta$. We will often slightly abuse terminology and say that $f$ has a particular bifurcation, without referring to the associated dynamical system $\dot y = f(y,\beta)$. 

There are often different equivalent formulations of bifurcation conditions. There is also often a natural order to checking the bifurcation, nondegeneracy and transversality conditions, namely if previous conditions are satisfied, then the computations required for the later conditions may be simplified. In this paper we make the choices which seem most appropriate for our purposes. 

In computations, especially those involving transversality, it is often convenient to replace the given parameters with other combinations of parameters. In this context, the following elementary result is useful.

\begin{proposition}[Reparameterisation]
\label[proposition]{prop:reparam}
Suppose that the system $\dot y = f(y,\beta)$ has some local bifurcation of equilibria $\mB$ at $(y_0, \beta_0)$. Let $k \leq m$ and $\Theta \colon V \to \R^k$ be any smooth map such that $\Theta_\beta(\beta_0)$ is surjective, and suppose that $f(y,\beta) \equiv \hat{f}(y, \Theta(\beta))$, where this equation defines $\hat{f}\colon U \times \Theta(V) \to \R^n$. Then the system $\dot y = \hat{f}(y,\theta)$ has the bifurcation $\mB$ at $(y_0, \theta_0)$, where $\theta_0:=\Theta(\beta_0)$ and the following are equivalent:
\begin{enumerate}
\item[(a)] The bifurcation in $\dot y = f(y,\beta)$ at $(y_0, \beta_0)$ is unfolded transversally by $\beta$;
\item[(b)] The bifurcation in $\dot y = \hat{f}(y,\theta)$ at $(y_0, \theta_0)$ is unfolded transversally by $\theta$.
\end{enumerate}
\end{proposition}
\begin{proof}
The claim that $\dot y = \hat{f}(y,\theta)$ has the bifurcation $\mB$ at $(y_0, \theta_0)$ follows trivially from our general assumptions about bifurcations, as the Taylor coefficients of $f$ at $(y_0, \beta_0)$ and $\hat{f}$ at  $(y_0, \theta_0)$ w.r.t. $y$ are identical to all orders. As noted in \cite{banaji:boros:hofbauer:2025}, associated with the bifurcation $\mB$ in $\dot y = f(y,\beta)$ is a ``bifurcation function'', say $F(y, \beta)$, such that in a neighbourhood of $(y_0, \beta_0)$, the bifurcation occurs if and only if $F(y,\beta)=0$. Further, the bifurcation is unfolded transversally by the parameters $\beta$ if and only if (i) $F_y(y_0, \beta_0)$ is injective and (ii) $\partial_{(y,\beta)} F(y_0, \beta_0)$ is surjective. Note that the bifurcation function $\wF(y,\theta)$ for the reparameterised system $\dot y = \hat{f}(y,\theta)$ is naturally defined via $F(y,\beta) = \wF(y,\Theta(\beta))$.

As $F_y(y_0, \beta_0) = \wF_y(y_0, \theta_0)$, the first condition for transversality holds in one case if and only if it holds in the other. Moreover, by the chain rule,
\[
[F_y\,|\, F_\beta] = [\wF_y\,|\,\wF_\theta]M \quad \mbox{where} \quad M=\left(\begin{array}{cc}I&0\\0& \Theta_\beta(\beta_0)\end{array}\right)\,,
\]
hence $\mathrm{rank}\,[F_y\,|\, F_\beta] \leq \mathrm{rank}[\wF_y\,|\,\wF_\theta]$. On the other hand, $M$ is a surjective matrix and hence has a right inverse, say $M^*$, giving $[F_y\,|\, F_\beta]M^* = [\wF_y\,|\,\wF_\theta]$; consequently, $\mathrm{rank}\,[F_y\,|\, F_\beta] \geq \mathrm{rank}[\wF_y\,|\,\wF_\theta]$. Thus $\mathrm{rank}\,[F_y\,|\, F_\beta] = \mathrm{rank}[\wF_y\,|\,\wF_\theta]$, and the second condition for transversality holds in one case if and only if it holds in the other. 
\end{proof}

\subsection{Basics on fold and cusp bifurcations}
We follow \cite{kuznetsov:2023} and define fold and cusp bifurcations in the system $\dot y = f(y, \beta)$ in terms of conditions on the Taylor expansion of $f$ about $p_0 = (y_0, \beta_0)$. We will also draw on material in \cite{golubitsky1985} in defining some of the conditions below. 

Assume that $f(p_0) = 0$, and define $J:= f_y(p_0)$. We always make the following basic assumption:
\begin{enumerate}
\item[B0.] $\ker J$ is one-dimensional. 
\end{enumerate}
Fix $0 \neq q \in\ker J$, and define the conditions:
\begin{enumerate}
\item[B1.] $p_0$ is a simply degenerate point of $f$ (equivalently, $J$ has one-dimensional centre subspace). 
\item[B2.] $B(q, q) \in \mathrm{im}\,J$;
\item[B3.] For some $i \in 1, \ldots, m$, $f_{\beta_i} \not \in \mathrm{im}\,J$; equivalently, $\mathrm{im}\,f_{\beta} \not \subseteq \mathrm{im}\,J$.
\end{enumerate}
Clearly B1 implies B0, and the condition B2 does not depend on the choice of $q$. Let B2${}^c$ refer to the negation of B2, namely the condition $B(q, q) \not \in \mathrm{im}\,J$. If B1 and B2${}^c$ hold, we say that $f$ has a {\bf nondegenerate fold bifurcation} at $p_0$. If, additionally, B3 holds, then $f$ has a {\bf nondegenerate fold bifurcation, unfolded by $\beta$}, at $p_0$. 

We will refer to a point satisfying B1 and B2 as a {\bf potential cusp point}. At any point satisfying B0 and B2, define:
\begin{enumerate}
\item[B4.] $C(q,q,q) - 3B(q, h) \not \in \mathrm{im}\,J$, where $h$ is any element of $J^{-1}(B(q,q))$.
\item[B5.] There exist $i, j \in \{1, \ldots, n\}$, such that for every $(K_1, K_2) \in \R^2\backslash\{0\}$, defining the differential operator $\mD := K_1 \partial_{\beta_i} + K_2\partial_{\beta_j}$, either (i) $\mD f \not \in \mathrm{im}\,J$; or (ii) $\mD f \in \mathrm{im}\,J$ and $\mD[f_y]q - B(q,h) \not \in \mathrm{im}\,J$, where $h$ is any element of $J^{-1}(\mD f)$.
\end{enumerate}
If B1, B2 and B4 hold, we say that $f$ has a {\bf nondegenerate cusp bifurcation} at $p_0$. If, additionally, B5 holds, we say that $f$ has a {\bf nondegenerate cusp bifurcation, unfolded by $\beta$,} at $p_0$.

\begin{remark}[basic bifurcation and nondegeneracy conditions]
Condition B1 is the basic bifurcation condition for a fold bifurcation, while B2${}^c$ is the nondegeneracy condition. Conditions B1 and B2 are the basic bifurcation conditions for a cusp bifurcation, while B4 is the nondegeneracy condition. 
\end{remark}

\begin{remark}[transversality conditions]
Conditions B3 and B5 are transversality conditions for nondegenerate fold and cusp bifurcations respectively. Condition B5 is more naturally stated via its negation: a pair of parameters $\beta_i, \beta_j$ {\em fails} to unfold a nondegenerate cusp bifurcation if and only if there exist $(K_1, K_2) \in \R^2\backslash\{0\}$ such that (with $\mD := K_1 \partial_{\beta_i} + K_2\partial_{\beta_j}$), (i) $\mD f \in \mathrm{im}\,J$; and (ii) $\mD[f_y]q - B(q,h) \in \mathrm{im}\,J$, where $h$ is any element of $J^{-1}(\mD f)$. While we have been unable to find condition B5 written down explicitly in this form, it can be derived either by examining computations in \cite{kuznetsov:2023} and \cite{golubitsky1985}; or by direct Taylor expansion of the vector field restricted to the parameter-dependent centre manifold near the bifurcation point.
\end{remark}

\begin{remark}[choice of preimages]
The particular choices of $h$ in Conditions B4 and B5 do not matter: assuming that $\ker J$ is one-dimensional, and B2 holds, given any vector $v \in \R^n$, and $h,h' \in J^{-1}v$, we have $h-h' = tq$ for some $t \in \R$ and hence $B(q,h) - B(q,h') = tB(q,q)\in \mathrm{im}\,J$. 
\end{remark}

\begin{remark}[geometric interpretation]
While the fold and cusp bifurcation conditions provide information on the local geometry of the zero-set of $f$ in $U \times V$ (the Cartesian product of state space and parameter space), the bifurcation conditions are not entirely geometrical: for example $(y,\beta) \mapsto \beta - y^2$ and $(y, \beta) \mapsto (\beta - y^2)^2$ have the same zero set in $\R^2$, but the first has a nondegenerate fold point at $(0,0)$, unfolded by the parameter $\beta$, while in the second case, the point $(0,0)$ is not a nondegenerate fold point, and the bifurcation at $(0,0)$ is not unfolded by $\beta$.
\end{remark}

Henceforth, if we refer to a fold (resp., cusp) bifurcation, without further qualification, this will mean a {\em nondegenerate} fold (resp., cusp) bifurcation, which is {\em unfolded} by the parameters of the system.

\subsection{Simplifying computations for fold and cusp bifurcations}
Before turning to the specific case of mass action networks, in the following four propositions, we list some simplifications which are useful for studying fold and cusp bifurcations in a general dynamical system. The proofs are matters of somewhat lengthy differentiation and so, in order not to distract from the main results, are outlined in Appendix~\ref{app:proofs}.

{\bf Changing coordinates.} As we would expect for any reasonably defined bifurcations, fold and cusp bifurcations are preserved under smooth, parameter-dependent, changes of coordinates. 

Let $U \subseteq \R^n,\,V \subseteq \R^m$ be open and let $f\colon U \times V \to \R^n$, $(y,\beta) \mapsto f(y,\beta)$ be smooth. Let $p_0:=(y_0, \beta_0) \in U \times V$ satisfy $f(p_0) = 0$. Given some $U' \subseteq \R^n$, consider any smooth map $\Psi: U' \times V \to U$ such that, for each $\beta \in V$, $\Psi(\cdot, \beta): U' \to U$ is a smooth diffeomorphism with inverse $\Psi^{-1}(\cdot, \beta)$. Define $\hat{f}(\hat{y},\beta):= \Psi^{-1}_y(\Psi(\hat{y},\beta ),\beta)f(\Psi(\hat{y},\beta),\beta)$. Define $\hat{y}_0:= \Psi^{-1}(y_0, \beta_0)$ and $\hat{p}_0:=(\hat{y}_0,\beta_0)$, so that $\hat{f}(\hat{p}_0) = 0$. 
\begin{proposition}
\label[proposition]{prop:coord}
Each condition B1--B5 holds for $f$ at $p_0$ if and only if the corresponding condition holds for $\hat{f}$ at $\hat{p}_0$. Consequently, the following are equivalent:
\begin{itemize}
\item[(a)] $f$ has a nondegenerate fold (resp., cusp) bifurcation unfolded by its parameters at $p_0$;
\item[(b)] $\hat{f}$ has a nondegenerate fold (resp., cusp) bifurcation unfolded by its parameters at $\hat{p}_0$.
\end{itemize}
\end{proposition}

{\bf Eliminating variables.}
The next proposition formalises the fact that, when studying fold and cusp bifurcations, we can use the implicit function theorem to eliminate variables (locally) in the spirit of Lyapunov-Schmidt reduction \cite{golubitsky1985}. Let $U_1 \subseteq \R^{n_1}$, $U_2 \subseteq \R^{n_2}$, $U_3 \subseteq \R^m$ be open, with $\hat{f}\colon U_1 \times U_2 \times U_3 \to \R^{n_1}$, $\hat{g}\colon U_1 \times U_2 \times U_3 \to \R^{n_2}$ smooth. Define $\wF:=(\hat{f},\hat{g})$. Assume that $p_0:=(x_0, y_0, \beta_0) \in U_1 \times U_2 \times U_3$ satisfies $\wF(p_0) = 0$, namely,
\begin{equation}
\label{eqpreproj}
\hat{f}(p_0)=0, \quad \hat{g}(p_0)=0\,.
\end{equation}
Now assume that $\hat{g}_y(p_0)$ is nonsingular, so that applying the implicit function theorem, there exists a smooth function $\phi(x,\beta)$ defined on some neighbourhood of $(x_0, \beta_0)$, with $y_0 = \phi(x_0, \beta_0)$ and $\hat{g}(x,\phi(x,\beta),\beta)\equiv 0$ on some neighbourhood $W$ of $p_0$. On $W$, define
\begin{equation}
\label{eqpostproj}
f(x,\beta):=\hat{f}(x,\phi(x,\beta), \beta), \quad g(x,y,\beta):= y - \phi(x,\beta)\, \quad \mbox{and} \quad F:= (f,g)\,,
\end{equation}
so that $F(p_0) = 0$. Assume, henceforth, that all derivatives of $F$ and $\wF$ are evaluated at $p_0$, while derivatives of $f$ are evaluated at $p_0':=(x_0, \beta_0)$. Define $\wJ := J^{(\wF)}$ and $J := J^{(F)}$. We easily confirm that $\dim \ker J = \dim \ker \wJ = \dim \ker J^{(f)}$. Assume that $\ker \wJ$ is one-dimensional, equivalently $\ker J$ is one-dimensional, equivalently $\ker J^{(f)}$ is one-dimensional. 
\begin{proposition}
\label[proposition]{prop:elimination}
Each condition B2--B5 holds for $\wF$ at $p_0$ if and only if the corresponding condition holds for $F$ at $p_0$ if and only if the corresponding condition holds for $f$ at $p_0'$. Consequently, if $p_0$ is a simply degenerate point of both $\wF$ and $F$, the following are equivalent:
\begin{itemize}
\item[(a)] $\wF$ has a nondegenerate fold (resp., cusp) bifurcation at $p_0$, unfolded by $\beta$;
\item[(b)] $F$ has a nondegenerate fold (resp., cusp) bifurcation at $p_0$, unfolded by $\beta$;
\item[(c)] $f$ has a nondegenerate fold (resp., cusp) bifurcation at $p_0'$, unfolded by $\beta$.
\end{itemize}
\end{proposition}

We remark that eliminating variables, effectively allowing us to work with a {\em projected} system of equations, makes the analysis of fold and cusp bifurcations simpler than many other bifurcations. Moreover, we will see that elimination becomes a particularly natural strategy when we work with the alternative equations for certain classes of networks, termed $\mP$-toric networks and discussed in Section~\ref{seccorollaries} below.

{\bf Applying linear mappings.} We are sometimes interested in applying to a vector field at each point some nonsingular parameter- and state-dependent linear map in order to simplify the equations. As a particular example, we frequently wish to clear denominators in a system of rational equations. Another example involves a canonical change of coordinates described in Section~\ref{secnatural}, leading to linear, parameter-dependent, factor in the differential equations which the next result allows us to disregard. 

Let $U \subseteq \R^n$, $V \subseteq \R^k$, $W \subseteq \R^s$ be open, and let $f \colon U \times V \to \R^n$ and $\hat{f} \colon U \times V \times W \to \R^n$ be smooth, with 
\[
\hat{f}(y,\beta,\gamma) = M(y,\beta,\gamma) f(y,\beta) \,, 
\]
where $M \colon U \times V \times W \to \R^{n \times n}$ is smooth (i.e., $M$ defines a family of square matrices smoothly dependent only $(y,\beta,\gamma)$). Let $\hat{p}_0:= (y_0, \beta_0, \gamma_0)$ and $p_0:= (y_0, \beta_0)$. Assume that $f(p_0) = 0$, hence, $\hat{f}(\hat{p}_0) = 0$; and that $M_0:= M(\hat{p}_0)$ is nonsingular. Clearly $\dim \ker J^{(f)} = \dim \ker J^{(\hat{f})}$. Assume that $\ker J^{(f)}$ is one-dimensional, equivalently, $\ker J^{(\hat{f})}$ is one-dimensional.

\begin{proposition}
\label[proposition]{prop:prefactor}
Each condition B2--B5 holds for $f$ at $p_0$ if and only if the corresponding condition holds for $\hat{f}(\cdot, \cdot, \gamma_0)$ at $p_0$, if and only if the corresponding condition holds for $\hat{f}$ at $\hat{p}_0$. Consequently, if $f$ and $\hat{f}$ have simply degenerate points at $p_0$ and $\hat{p}_0$ respectively, then the following are equivalent:
\begin{itemize}
\item[(a)] $f$ has a nondegenerate fold (resp., cusp) bifurcation unfolded by its parameters at $p_0$;
\item[(b)] $\hat{f}$ has a nondegenerate fold (resp., cusp) bifurcation unfolded by its parameters $\beta$ at $\hat{p}_0$.
\end{itemize}
\end{proposition}
We remark on the particular consequence of \Cref{prop:prefactor} that the parameters $\gamma$ can play no part in unfolding fold and cusp bifurcations in $\hat{f}(y,\beta,\gamma)$.

{\bf Applying a diffeomorphism to both sides of equations.} Our final example of an operation which can help simplify an equation prior to analysis, involves applying a local diffeomorphism to both sides of the equation. For example, raising both sides to a positive power is an operation which can help simplify ``posynomial equations'', namely, generalised polynomial equations with real powers \cite{boyd2004convex}, which can arise as the alternative equations constructed below, see also \cite[Lemma 3.17]{banajifeliu26}.

Let $U \subseteq \R^n$, $V \subseteq \R^m$ be open, let $F,G \colon U \times V \to \R^n$ be smooth, and define 
\[
f(y,\beta)=F(y,\beta)-G(y,\beta)\,. 
\]
Let $p_0 = (y_0, \beta_0) \in U \times V$ satisfy $f(p_0) = 0$, and define $y_0' = F(p_0) = G(p_0)$. Given some open neighbourhood $U' \subseteq \R^n$ of $y_0'$, let $\Psi \colon U' \times V \to \R^n$ be smooth with $M_0:=\Psi_y(y_0',\beta_0)$ invertible. Define $\hat{f} \colon F^{-1}(U') \cap G^{-1}(U') \to \R^n$ by 
\[
\hat{f}(y,\beta) = \Psi(F(y,\beta),\beta) - \Psi(G(y,\beta),\beta)\,.
\]
Clearly, $J^{(\hat{f})} = M_0J^{(f)}$, hence $\ker J^{(\hat{f})} = \ker J^{(f)}$. Assume that $\ker J^{(f)}$ is one-dimensional, equivalently, $\ker J^{(\hat{f})}$ is one-dimensional.

\begin{proposition}
  \label[proposition]{prop:power}
Each condition B2--B5 holds for $\hat{f}$ at $p_0$ if and only if the corresponding condition holds for $f$. Consequently, if $f$ and $\hat{f}$ both have simply degenerate points at $p_0$, then the following are equivalent:
\begin{itemize}
\item[(a)] $f$ has a nondegenerate fold (resp., cusp) bifurcation unfolded by its parameters at $p_0$;
\item[(b)] $\hat{f}$ has a nondegenerate fold (resp., cusp) bifurcation unfolded by its parameters at $p_0$.
\end{itemize}
\end{proposition}

With these preliminary results in place, we are in a position to approach mass action networks. 

\section{Mass action systems: original and alternative equations}
\label{secMA}

We recall and build on some material from \cite{BBH2024smallbif,banajifeliu26}, with some notation altered for readability and for consistency with the treatment here. 

\subsection{Basic definitions on mass action networks}
A chemical reaction converts a formal linear combination of chemical species, termed the {\bf source} of the reaction, into another termed the {\bf product} of the reaction. A mass action network will refer to a system of chemical reactions with mass action kinetics. The ordinary differential equation (ODE) system associated with any mass action network involving $n$ species and $m$ reactions takes the form
\begin{equation}
\label{eqODE}
\dot x = g(x,\k):=\Gamma(\k \circ x^A)
\end{equation}
where $x \in \mathbb{R}^n_{\geq 0}$ is the vector of species concentrations; $\Gamma \in \mathbb{Z}^{n \times m}$ is the {\bf stoichiometric matrix} whose $(i,j)$th entry is the net production of species $i$ in reaction $j$; $A \in \mathbb{Z}_{\geq 0}^{m \times n}$ is the {\bf exponent matrix} whose $(i,j)$th entry is the stoichiometry of species $j$ in the source of the $i$th reaction; $x^A$ denotes the vector of $m$ monomials whose $i$th entry is $\prod_{j=1}^n x_j^{a_{ij}}$; and $\k \in \mathbb{R}^m_+$ is a vector of {\bf rate constants}. The {\bf rank} of the network is the rank of $\Gamma$, and we refer to an $n$-species, $m$-reaction network with rank $r$ as an $(n,m,r)$ network. When $r=n$, we refer to the network as a {\bf full-rank} network. The network is termed {\bf quadratic} if sources are at most bimolecular, equivalently, row sums of $A$ are at most $2$, namely, all polynomials in $x^A$ are at most quadratic. The network is {\bf bimolecular} if both sources and products are at most bimolecular, namely row sums of $A$ and $\Gamma + A^\top$ are all at most $2$. 

For later use, we note the form of the Jacobian matrices,
\begin{equation}
\label{eqJstarg}
g_x(x_0, \k_0) = \Gamma D_{\k_0 \circ x_0^A} A D_{1/x_0} \quad \mbox{and} \quad g_{\k}(x_0, \k_0) = \Gamma D_{\k_0 \circ x_0^A} D_{1/\k_0}\,.
\end{equation}

We fix an arbitrary $(n,m,r)$ network with stoichiometric matrix $\Gamma$ and exponent matrix $A$. Given $x \in \R^n_+$, the set $(x+\mathrm{im}\,\Gamma) \cap \R^n_+$ is termed the {\bf positive stoichiometric class} of $x$. Positive stoichiometric classes are easily seen to be locally invariant under \eqref{eqODE}, and when $r<n$ are of the form
\begin{equation}\label{eq:Z1}
S_K:= \{x\in \R^n_+\,:\,Z x -K=0\}\,,
\end{equation}
where $Z\in \R^{(n-r) \times n}$ is a {\bf matrix of conservation laws} whose rows form a basis of $\ker\Gamma^\top$, and $K\in \R^{n-r}$ is a vector of real constants. Note that $Z$ is not uniquely defined; however, the notation $S_K$ assumes that a choice of $Z$ has been made. For given $\k \in \R^m_+$ and $K \in \R^{n-r}$, positive equilibria on $S_K$ are the solutions in $\R^n_+$ to 
\[
g(x,\k) = 0, \quad Zx-K = 0\,.
\]
Let $(x_0, \k_0) \in \R^n_+ \times \R^m_+$ satisfy $g(x_0,\k_0) = 0$. When considering local bifurcations at $(x_0, \k_0)$, we are interested in the system $\dot x = g(x,\k)$ {\em restricted} to the positive stoichiometric class of $x_0$, namely to $S_K$ where $K = Zx_0$. We could, if desired, choose coordinates on $\mathrm{im}\,\Gamma$, and write down this restricted system explicitly in terms of these coordinates. We will not need to do this, but will sometimes write $\tilde{g}_K\colon S_K \times \R^m_+ \to \mathrm{im}\,\Gamma$ to denote the corresponding restricted vector field. We will also abuse terminology and write that \eqref{eqODE} has some bifurcation at $(x_0, \k_0)$ to mean that $\dot x = \tilde{g}_K(x,\k)$ ($K= Zx_0$) has this bifurcation. In terms of Taylor coefficients, restricting attention to $S_K$ means that we are interested in the action of the derivative $J^{(g)}:=g_x(x_0,\k_0)$ on $\mathrm{im}\,\Gamma$, and more generally in the action of the $k$th derivative of $g$ w.r.t. $x$, evaluated at $(x_0, \k_0)$, on 
\[
\underbrace{\mathrm{im}\,\Gamma \times \cdots \times \mathrm{im}\,\Gamma}_{k \text{ times}}\,.
\]
In particular, we say that $x_0$ is a {\bf degenerate} equilibrium of \eqref{eqODE} if $\ker (\tilde{g}_K)_x(x_0, \k_0) = \ker J^{(g)} \cap \mathrm{im}\,\Gamma$ is nontrivial; and $x_0$ is {\bf simply degenerate} if $(\tilde{g}_K)_x(x_0, \k_0)$ has a zero eigenvalue of algebraic multiplicity $1$, and no other eigenvalues on the imaginary axis. We refer to a network as {\bf degenerate} if all of its positive equilibria are degenerate. 

\begin{remark}[the dimension of $\ker J^{(g)}$]
\label[remark]{simplydegen}
We make a general observation about $J^{(g)}$ evaluated at any equilibrium at which $\ker J^{(g)} \cap \mathrm{im}\,\Gamma$ is one-dimensional. At such a point $J^{(g)}$ may have rank $r$ or $r-1$. In the case that $\mathrm{rank}\,J^{(g)} = r$, we have that $\ker J^{(g)}$ and $\mathrm{im}\,\Gamma$ are $n-r$ and $r$ dimensional subspaces of $\R^n$ with one-dimensional intersection, hence $\ker J^{(g)} + \mathrm{im}\,\Gamma$ does not span $\R^n$; if $\mathrm{rank}\,J^{(g)} = r-1$, $\ker J^{(g)}$ and $\mathrm{im}\,\Gamma$ are $n-r+1$ and $r$ dimensional subspaces of $\R^n$ with one-dimensional intersection and hence $\ker J^{(g)} + \mathrm{im}\,\Gamma$ spans $\R^n$. (This also makes it clear why $\mathrm{rank}\,J^{(g)}$ can be no less than $r-1$.)
\end{remark}

\subsection{Partitions and matrices associated with a network}
A ``{\bf $p$-partition}'' $\mP$ of $\{1, \ldots, m\}$ will mean a set of subsets $P_1, \ldots, P_p$ of $\{1, \ldots, m\}$ such that $\{1, \ldots, m\} = \sqcup_{i=0}^p P_i$. Associated with each block $P_i$ of $\mP$ is the coordinate hyperplane $(\R^n)_{P_i} = \{x \in \R^n\,:\,x_j = 0\,\,\mbox{for all}\,\, j \not \in P_i\}$. Given $X \subseteq \mathbb{R}^m$ or $X \in \R^m$, we write $X_{P_i}$ for the orthogonal projection of $X$ onto $(\R^n)_{P_i}$; and $\bm{1}_\mP$ will denote the $m \times p$ matrix whose $i$th column is $\bm{1}_{P_i}$. Any linear subspace $\mH \subseteq \R^m$ is said to {\bf factor over $\mP$} if it has a direct sum decomposition of the form 
\[
\mH = \mH_{P_1} \oplus \cdots \oplus \mH_{P_p}\,.
\]
We say that the reaction network {\bf factors over} $\mP$ if $\ker \Gamma$ factors over $\mP$. Observe that if a reaction network factors over $\mP$, then $\Gamma z = 0 \Rightarrow \Gamma ((\bm{1}_{\mP}t) \circ z) = 0$ for any $t \in \R^p$. In particular, if $\Gamma ((\bm{1}_{\mP}t) \circ z) = 0$ for some $t \in \R^p_+$, then $\Gamma z = 0$, an observation we will use frequently.

From now on, we assume that our fixed $(n,m,r)$ network of interest factors over a $p$-partition $\mP$ (note that $\mP$ may be the trivial partition, so this assumption is without loss of generality). Define
\[
s_{\mP}:=\mathrm{rank}\,[A\,|\,\bm{1}_\mP] - p\,.
\]
Clearly $s_{\mP} \leq n$ and $s_{\mP} \leq m-p$. Networks with $s_{\mP} < r$ are termed {\bf $\mP$-overdetermined} and, in \cite[Theorem~4.2(iii)]{banajifeliu26}, it is shown that these networks are degenerate. In particular, such networks admit no generic local bifurcations of positive equilibria and are, from the point of view of this paper, uninteresting. 

Networks with $s_{\mP} = r$ are termed {\bf $\mP$-toric}. We will see that in $\mP$-toric networks, the alternative equations take a block triangular form, and hence these networks are particularly amenable to the analysis to follow. 

In addition to the matrix $Z$ of conservation laws, we will be interested in the following three matrices which depend on the network and the partition $\mP$ and are not, in general, uniquely defined. 
\begin{enumerate}
\item An $(m-p-s_{\mP})\times m$ {\bf solvability matrix} $W$ whose rows are a basis of $\ker\,[A\,|\,\bm{1}_\mP]^\top$.

\smallskip
\item An $(n+p) \times (n-s_{\mP})$ full rank matrix $Q$ whose columns are a basis of $\ker\,[A\,|\,\bm{1}_\mP]$. Also define $\wQ$ to be the $n \times (n-s_{\mP})$ matrix consisting of the first $n$ rows of $Q$.

\smallskip
\item An $(n+p) \times m$ matrix $G$ which is a \textbf{generalised inverse} of $[A\,|\,\bm{1}_\mP]$ i.e., such that 
\[ [A\,|\,\bm{1}_\mP]\,G\,[A\,|\,\bm{1}_\mP] = [A\,|\,\bm{1}_\mP]\, . \] 
Also define $\wG$ to be the $n \times m$ matrix consisting of the first $n$ rows of $G$. 

\end{enumerate}
Note that $W$, $Q$ and $Z$ may be empty, and various natural conventions, detailed in \cite{banajifeliu26}, allow us to deal with these cases (see also \Cref{remspecial} below).

\subsection{Natural coordinates}
\label{secnatural}

In \cite[Section 2.5]{BBH2024smallbif}, we presented a change of coordinates applicable to full-rank networks which proved useful when studying bifurcations, not necessarily fold or cusp. We digress, slightly, to present a generalisation of this construction to networks which are not necessarily full-rank, and which may factor over a nontrivial partition. 

With $G$ and $W$ defined as above, observe that 
\[
\overline{G} := \left(\begin{array}{c}G\\W\end{array}\right)
\]
has full column rank, namely rank $m$. This follows since, by the definitions of $G$ and $W$, $([A\,|\,\bm{1}_{\mP}]G-I)[A\,|\,\bm{1}_{\mP}]=0$, hence $[A\,|\,\bm{1}_{\mP}]G-I = -UW$ for some matrix $U$; consequently $[A\,|\,\bm{1}_{\mP}\,|\,U]\overline{G} = I$, namely $\overline{G}$ has left inverse $[A\,|\,\bm{1}_{\mP}\,|\,U]$. With $\wG$ referring to the top $n \times m$ block of $G$, and $V$ comprising its next $p$ rows, this construction gives
\begin{equation}\label{eq:identity}
A\wG + \bm{1}_\mP V + UW = I\,, \quad \mbox{namely,} \quad I-A\wG=\bm{1}_\mP V + UW\,.
\end{equation}
Now define {\bf natural coordinates} $y$ on $\R^n_+$ by $y = \k^{\wG}\,\circ x$. We obtain the ODE for $y$ as follows:
\begin{eqnarray*}
  \dot y = \k^{\wG} \circ \dot x
         & = & D_{\k^{\wG}}\,\Gamma\,(\k \circ (\k^{-\wG}\,\circ y)^A)\\
         & = & D_{\k^{\wG}}\,\Gamma\,(\k^{I-A\wG} \circ y^A)\\
         & = & D_{\k^{\wG}}\,\Gamma\,(\k^{\bm{1}_{\mP}\,V + UW} \circ y^A) \qquad \mbox{(using \eqref{eq:identity})}\\
  & = & D_{\k^{\wG}}\,\Gamma\,((\k^V)^{\bm{1}_{\mP}} \circ (\k^W)^U \circ y^A)\\
  & = & D_{\k^{\wG}}\,\Gamma D_{\bm{1}_{\mP}\k^V}\,((\k^W)^U \circ y^A)\\
    & = & D_{\k^{\wG}}\,\Gamma D_{\bm{1}_{\mP}\k^V} \Gamma^* \Gamma \,((\k^W)^U \circ y^A)\\
    &=& M(\kappa) \Gamma \,((\k^W)^U \circ y^A)\,.
\end{eqnarray*}
Here $\Gamma^*$ is any generalised inverse of $\Gamma$, and the final equation defines the parameter-dependent $n \times n$ matrix $M(\kappa):= D_{\k^{\wG}}\Gamma D_{\bm{1}_{\mP}\k^V} \Gamma^*$. We have also used the fact that $\Gamma D_{\bm{1}_{\mP}\k^V} = \Gamma D_{\bm{1}_{\mP}\k^V} \Gamma^* \Gamma$. To justify this, note first that by definition $\Gamma \Gamma^* \Gamma = \Gamma$ and hence, $\mathrm{im}(\Gamma^* \Gamma - I) \in \ker \Gamma$. As $\ker \Gamma$ factors over $\mP$, it follows that $\mathrm{im}(D_{\bm{1}_{\mP}\k^V} \Gamma^* \Gamma - D_{\bm{1}_{\mP}\k^V}) \in \ker \Gamma$. 

We observe that $M(\kappa):=D_{\k^{\wG}}\Gamma D_{\bm{1}_{\mP}\k^V} \Gamma^*$ acts as a nonsingular transformation on $\mathrm{im}\,\Gamma$. To see this, note that:
\[
D_{\k^{\wG}}\Gamma D_{\bm{1}_{\mP}\k^V} \Gamma^*\Gamma v = 0\,\,\Rightarrow\,\, \Gamma D_{\bm{1}_{\mP}\k^V} \Gamma^*\Gamma v = 0\,\,\Rightarrow\,\, \Gamma \Gamma^*\Gamma v =0 \,\,\Rightarrow\,\, \Gamma v =0\,,
\]
where the first implication follows as $D_{\k^{\wG}}$ is nonsingular, the second follows as $\ker \Gamma$ factors over $\mP$, and the third follows from the definition of a generalised inverse.

Note that $\k^W$ is a vector of $m-s_\mP-p$ new positive parameters which we will refer to as the {\em inner parameters}. On the other hand, $M(\kappa)$ is a matrix dependent on the $n+p$ {\em outer parameters} $\k^{\wG}, \k^V$. It is easily seen that we can reduce the number of parameters by one by incorporating a parameter from $\k^V$ into $\k^{\wG}$. We may subsequently rescale time to reduce the number of outer parameters to $n+p-2$ leaving an $(m+n-s_\mP-2)$-parameter family of ODEs. 
Moreover, only the $m-s_\mP-p$ inner parameters directly affect the equilibrium set; the remaining parameters act via a positive linear transformation on the whole vector field.

\begin{proposition}[Consequences for full rank networks]
\label[proposition]{prop:natural}
Consider an $(n,m,n)$ network which factors over a $p$-partition $\mP$ giving rise to the mass action ODE $\dot x = g(x,\k) := \Gamma(\k\circ x^A)$. Consider also the system $\dot y = \hat{g}(y,\hat{\k}) := \Gamma(\hat{\k}^U\circ y^A)$. Given $(x_0,\k_0)$, define $y_0 = \k_0^{\wG} \circ x_0$ and $\hat{\k}_0 = \k_0^W$. Then
\begin{itemize}
\item[(a)] $g(x_0,\k_0) = 0$ if and only if $\hat{g}(y_0,\hat{\k}_0) = 0$. 
\item[(b)] $\ker J^{(g)}(x_0,\k_0) = \ker J^{(\hat{g})}(y_0,\hat{\k}_0)$. 
\end{itemize}
We now assume that $g(x_0,\k_0) = 0$, equivalently $\hat{g}(y_0,\hat{\k}_0) = 0$, and that $\ker J^{(g)}(x_0,\k_0)$ is one-dimensional, equivalently, $\ker J^{(\hat{g})}(y_0,\hat{\k}_0)$ is one-dimensional. Then
\begin{itemize}
\item[(c)] Each of conditions B2--B5 holds for $\dot x = g(x,\k)$ at $(x_0, \k_0)$ if and only if the corresponding condition holds for $\dot y = \hat{g}(y,\hat{\k})$ at $(y_0, \hat{\k}_0)$. In particular if $\dot x = g(x,\k)$ has a simply degenerate equilibrium at $(x_0, \k_0)$, and the same holds for $\dot y = \hat{g}(y,\hat{\k})$ at $(y_0, \hat{\k}_0)$, then $\dot x = g(x,\k)$ has a fold (resp., cusp) bifurcation at $(x_0, \k_0)$ if and only if $\dot y = \hat{g}(y,\hat{\k})$ has a fold (resp., cusp) bifurcation at $(y_0, \hat{\k}_0)$.
\item[(d)] If $m\leq n+p$ (resp., $m \leq n+p+1$) the network forbids a fold (resp., cusp) bifurcation.
\end{itemize}
\end{proposition}
\begin{proof}
(a) and (b). From above, $D_{\k_0^{\wG}} g(x_0, \k_0) = M(\k_0)\hat{g}(y_0, \hat{\k}_0)$; similarly $D_{\k_0^{\wG}} J^{(g)}(x_0, \k_0) = M(\k_0)J^{(\hat{g})}(y_0, \hat{\k}_0)$. The claims now follow as both $D_{\k_0^{\wG}}$ and $M(\k_0)$ represent nonsingular linear transformations.

(c) From \Cref{prop:coord}, each condition B1--B5 holds for $\dot x = g(x,\k):= \Gamma(\k\circ x^A)$ at $(x_0, \k_0)$ if and only if it holds for $\dot y = M(\kappa) \Gamma((\k^W)^U\circ y^A)$ at $(y_0, \k_0)$. As $M(\k_0)$ is nonsingular, by \Cref{prop:prefactor}, each condition B2--B5 holds for $\dot y = M(\kappa) \Gamma((\k^W)^U\circ y^A)$ at $(y_0, \k_0)$ if and only if it holds for $\dot y = \Gamma((\k^W)^U\circ y^A)$ at $(y_0, \k_0)$. As the map $\R^m_+ \to \R^{m-p-s_{\mP}}_+$, $\k \mapsto \k^W$ has surjective derivative, \Cref{prop:reparam} now allows us to write $\hat{\k}:= \k^W$ to obtain the claim. 

(d) If the network is $\mP$-overdetermined, hence degenerate, the final claim is trivial. Otherwise $\hat{\k}$ consists of $m-n-p = 0$ (resp., $m-n-p = 1$) parameters, and $\dot y = \Gamma(\hat{\k}^U\circ y^A)$ forbids a fully unfolded fold (resp., cusp) bifurcation. In particular, condition B3 (resp., B5) trivially fails for $\dot y = \Gamma(\hat{\k}^U\circ y^A)$ at every potential fold (resp., cusp) point. By \Cref{prop:prefactor}, the same holds for the original mass action system. 
\end{proof}

\begin{remark}
By reducing the number of parameters we need to consider from $m$ to $m-p-s_{\mP}$, \Cref{prop:natural}(c) can greatly simplify the study of fold and cusp bifurcations in full-rank mass action networks. The claim in \Cref{prop:natural}(d) will also follow from the results below, see \Cref{propfull1}. A special case of the claim about cusp bifurcations in \Cref{prop:natural}(d) for the trivial partition appeared in \cite[Theorem 36]{BBH2024smallbif}.
\end{remark}

We return to the network presented in \Cref{Ex363a} in order to illustrate natural coordinates, and the results in \Cref{prop:natural}. 
\begin{ex}[Revisiting \Cref{Ex363a}]
\label[example]{Ex363b}
Consider again the following full-rank $3$-species, $6$-reaction network:  
\[
\fX \ce{->[\k_1]} 2\fZ\,,\quad 2\fZ \ce{->[\k_2]} \fX + \fZ,\quad \fX + \fZ \ce{->[\k_3]} \fX\,,\quad \fY+\fZ \ce{->[\k_4]} \fY\,,\quad \fX+\fY \ce{->[\k_5]} \fZ\,,\quad \fZ \ce{->[\k_6]} \fX+\fY\,.
\]
It is easily confirmed that the network factors over the nontrivial partition $\mP:=\{\{1,2,3,4\},\{5,6\}\}$, and thus, as $m=6, n=3$ and $p=2$, \Cref{prop:natural}(b) tells us that this network forbids cusp bifurcations. To illustrate the ideas, we derive the ODEs in natural coordinates. We have
\[
\Gamma=\left(\begin{array}{rrrrrr}-1&1&0&0&-1&1\\0&0&0&0&-1&1\\2&-1&-1&-1&1&-1\end{array}\right), \quad [A\,|\,\bm{1}_{\mP}] = \left(\begin{array}{cccccc}1&0&0&1&0\\0&0&2&1&0\\1&0&1&1&0\\0&1&1&1&0\\1&1&0&0&1\\0&0&1&0&1\end{array}\right)\,.
\]
The original mass action system has the form
\[
\left(\begin{array}{c}\dot x\\\dot y \\\dot z \end{array}\right) = \Gamma (\k \circ x^A) = \left(\begin{array}{rrrrrr}-1&1&0&0&-1&1\\0&0&0&0&-1&1\\2&-1&-1&-1&1&-1\end{array}\right)\left(\begin{array}{c}\k_1x\\\k_2z^2\\\k_3xz\\\k_4yz\\\k_5xy\\\k_6z\end{array}\right)\,.
\]
We can choose $W = (1,2,-2,-1,1,-1)$, and
\[
\Gamma^*=\left(\begin{array}{rrr}0&0&0\\1&-1&0\\-1&0&-1\\0&0&0\\0&0&0\\0&1&0\end{array}\right)\,, \quad G = \left(\begin{array}{rrrrrr}-1&-1&2&0&0&0\\-1&-1&1&1&0&0\\-1&0&1&0&0&0\\2&1&-2&0&0&0\\1&0&-1&0&0&1\end{array}\right)\,, \quad U = \left(\begin{array}{r}0\\0\\0\\0\\1\\0\end{array}\right)\,.
\]
In natural coordinates, 
\[
\left(\begin{array}{c}X\\Y\\Z\end{array}\right) = \k^{\wG}\circ \left(\begin{array}{c}x\\y\\z\end{array}\right)\,,
\]
we get, after some manipulation,
\[
\left(\begin{array}{c}\dot X\\\dot Y \\\dot Z \end{array}\right) 
= \left(\begin{array}{ccc}\t_3\t_2&\t_3(1-\t_2)&0\\0&\t_4&0\\0&\t_5(\t_2-1)&\t_5\t_2\end{array}\right) \left(\begin{array}{rrrrrr}-1&1&0&0&-1&1\\0&0&0&0&-1&1\\2&-1&-1&-1&1&-1\end{array}\right)\left(\begin{array}{c}X\\Z^2\\XZ\\YZ\\\t_1XY\\Z\end{array}\right)
\]
where $\t_1 := \k^{W} = \frac{\k_2^2\k_1\k_5}{\k_3^2\k_4\k_6}$, is the sole inner parameter in the system, and $\t_2 = \frac{\k_1\k_2}{\k_3\k_6}$, $\t_3 = \frac{\k_3\k_6}{\k_2}$, $\t_4 = \frac{\k_4\k_6}{\k_2}$, $\t_5 = \k_6$. (If desired, a rescaling of time allows us to discard one of $\t_3, \t_4$ or $\t_5$, leaving a total of $4$ parameters.) The only parameter combination which affects the equilibrium set is $\t_1$, and by \Cref{prop:prefactor}, the parameters $\t_2, \ldots, \t_5$ cannot participate in unfolding a fold or cusp bifurcation. Thus the system forbids a (fully unfolded) cusp bifurcation. 
\end{ex}

\subsection{The alternative equations}

Define $\mC:= \ker \Gamma \cap \R^m_+$ to be the {\bf positive flux cone} of the network, which we assume to be nonempty (otherwise the network admits no positive equilibria). Write $\mC_i:=\mC_{P_i}$ so that $\mC = \mC_1 + \cdots + \mC_{p}$. For each $i = 1, \ldots, p$, choose any flat cross section of $\mC_i$, as described in \cite{banajifeliu26}, say $\widehat{\mC}_i$, and choose any affine bijection 
\[h^{(i)} \colon Y_i \longrightarrow \widehat{\mC_i}\]
with domain $Y_i \subseteq \mathbb{R}^{\dim \mC_i-1}$ an open polytope. Writing 
\[ Y := Y_1 \times \cdots \times Y_p \subseteq \mathbb{R}^{m-r-p},\quad  \quad \widehat{\mC}:= \widehat{\mC}_1 + \cdots + \widehat{\mC}_p \subseteq \R^m_+\, , \]
and 
$\alpha := (\alpha^{(1)}, \ldots, \alpha^{(p)})$ for an element of $Y$,
we define the affine bijection
\begin{equation}\label{eq:h}
h \colon \quad Y \longrightarrow  \  \widehat{\mC}  \, \qquad 
(\alpha^{(1)}, \ldots, \alpha^{(p)}) \mapsto  \sum_{i=1}^p\,h^{(i)}(\alpha^{(i)}) \,. 
\end{equation}
Observe that $\mC = \{(\bm{1}_{\mP}\lambda) \circ \hat{c}\,:\, \lambda \in \R^p_+, \hat{c} \in \widehat{\mC}\} = \{(\bm{1}_{\mP}\lambda) \circ h(\alpha)\,:\, \lambda \in \R^p_+, \alpha \in Y\}$. We write $H$ for the (constant) linear map $h_y$, and note that $\mathrm{im}\,H \subseteq \ker \Gamma$ and $\mathrm{im}\,H$ factors over $\mP$ (see \cite[Lemma 3.5]{banajifeliu26}). Clearly, given $(x_0, \k_0) \in \R^n_+ \times \R^m_+$ satisfying $g(x_0, \k_0)=0$, there exists a uniquely defined $(\hat{c}_0, \lambda_0) \in \widehat{\mC} \times \R^p_+$ such that $\k_0 \circ x_0^A = (\bm{1}_{\mP}\lambda_0) \circ \hat{c}_0$. 

We will frequently call on the following basic result from \cite{banajifeliu26}. 
\begin{lemma}
\label[lemma]{lemGsolve}
Consider an $(n,m,r)$ network which factors over a $p$-partition $\mP$.
\begin{itemize}
\item[(i)] Given $u \in \mathrm{im}\,[A\,|\,\bm{1}_\mP]$, $v \in \R^{n-s_{\mP}}$, and $z \in \R^n$, satisfying $\wG u + \wQ v = z$, there exists $t \in \R^p$ such that $u = A z  + \bm{1}_\mP t$. 
\item[(ii)] Given $(z,t) \in \R^n \times \R^p$, there exists a unique $v \in \R^{n-s_{\mP}}$ such that $\wG(A z  + \bm{1}_\mP t) + \wQ v = z$.
\end{itemize}
Consequently, given $(z,t) \in \R^n \times \R^p$, the pair of equations $Wu=0$, $\wG u + \wQ v = z$ has a unique solution $(u(z,t),v(z,t)) \in \R^m \times \R^{n-s_{\mP}}$ with $u(z,t) = Az + \bm{1}_{\mP} t$.  
\end{lemma}
\begin{proof}
Parts (i) and (ii) are proved in \cite[Lemma~3.12]{banajifeliu26}, while the final claim is an immediate consequence of the first two. 
\end{proof}

Define $F\colon Y \times \R^{n-s_\mP}_+ \times \R^m_+ \to \R^n$ by
\begin{equation}
\label{eq:F}
  F(\alpha,\mu,\k) := (h(\alpha)/\k)^{\wG} \circ \mu^{\wQ}\,.
\end{equation}
For some $(\alpha_0, \mu_0, \k_0)$, let $x_0:= F(\alpha_0,\mu_0,\k_0)$. We compute
\begin{equation}
\label{eqJF}
\partial_{(\alpha, \mu)} F(\alpha_0, \mu_0, \k_0) = \left(\begin{array}{cc}D_{x_0}\wG D_{1/h(\alpha_0)}H& D_{x_0}\wQ D_{1/\mu_0} \end{array}\right)\,.
\end{equation}
\begin{comment}
and
\begin{equation}
\label{eqJhatF}
\partial_{(\alpha, \mu, \k)}\wF(\alpha_0, \mu_0, \k_0)= \left(\begin{array}{ccc}D_{x_0}\wG D_{1/h(\alpha_0)}H & D_{x_0} \wQ D_{1/\mu_0} & -D_{x_0} \wG D_{1/\k_0}\\0&0&I\end{array}\right)\,.
\end{equation}
\end{comment}
Given some fixed $K \in \R^{n-r}$, define $f_K\colon Y \times \R^{n-s_\mP}_+ \times \R^m_+ \to \R^{m-s_\mP-p} \times \R^{n-r}$ by
\begin{equation}
\label{eqalt}
f_K(\alpha, \mu, \k) := \left(\begin{array}{c}h(\alpha)^W - \k^W\\Z F(\alpha, \mu, \k) - K\end{array}\right) = \left(\begin{array}{c}h(\alpha)^W - \k^W\\Z[(h(\alpha)/\k)^{\wG}\circ \mu^{\wQ}] - K\end{array}\right)\,.
\end{equation}
We refer to the system of equations 
\[
f_K(\alpha, \mu, \k)=0
\]
as the {\bf alternative equations} for the network. The subsystem $h(\alpha)^W - \k^W = 0$ is referred to as the {\bf solvability system} for the network.

We calculate: 
\begin{equation}
\label{eqJf}
\partial_{(\alpha, \mu)}f_K(\alpha_0, \mu_0, \k_0) = \left(\begin{array}{cc}D_{h(\alpha_0)^W}WD_{1/h(\alpha_0)}H&0\\ZD_{x_0}\wG D_{1/h(\alpha_0)}H& ZD_{x_0}\wQ D_{1/\mu_0} \end{array}\right)\,,
\end{equation}
and
\begin{comment}
\begin{equation}
\label{eqJfstar}
\partial_{(\alpha, \mu, \k)}f_K(\alpha_0, \mu_0, \k_0) = \left(\begin{array}{ccc}D_{h(\alpha_0)^W}WD_{1/h(\alpha_0)}H&0& - D_{\k_0^W}WD_{1/\k_0}\\ZD_{x_0}\wG D_{1/h(\alpha_0)}H & ZD_{x_0} \wQ D_{1/\mu_0} & -ZD_{x_0} \wG D_{1/\k_0}\end{array}\right)\,.
\end{equation}
\end{comment}
\begin{equation}
\label{eqJfstar}
\partial_{\k}f_K(\alpha_0, \mu_0, \k_0) = -\left(\begin{array}{c}D_{\k_0^W}W\\ZD_{x_0} \wG \end{array}\right) D_{1/\k_0}\,.
\end{equation}

A key result from \cite{banajifeliu26} that we will call on is the following:
\begin{theorem}
\label[theorem]{thm:degensolvability}
Consider an $(n,m,r)$ network which factors over a $p$-partition $\mP$. Fix $\k_0 \in \R^m_+$ and $K \in \R^{n-r}$. 
\begin{enumerate}[align=left,leftmargin=*, label=(\roman*)]
\item[(i)] $F(\cdot, \cdot, \k_0)$ is a smooth bijection between the zero set of $f_K(\cdot, \cdot, \k_0)$ and the set of positive equilibria of the network on $S_K$. 
\end{enumerate} 
Let $x_0 \in \R^n_+$ be an equilibrium on $S_K$, i.e., $g(x_0,\k_0) = 0$ and $Zx_0 = K$. Using (i), define $(\alpha_0, \mu_0) \in Y \times \mathbb{R}_+^{n-s_{\mP}}$ to be the (unique) point satisfying $f_K(\alpha_0, \mu_0, \k_0) = 0$ and $F(\alpha_0, \mu_0, \k_0) = x_0$. Write $J^{(f)}:=\partial_{(\alpha, \mu)}f_K(\alpha_0, \mu_0,\k_0)$, $J^{(g)}:= \partial_{x}g(x_0, \k_0)$. 
\begin{enumerate}[align=left,leftmargin=*, label=(\roman*)]
\item[(ii)] Then 
\[\ker J^{(f)} \cong \ker J^{(g)} \cap \mathrm{im}\,\Gamma\, ,\] 
with the isomorphism given by $J^{(F)} := \partial_{(\alpha, \mu)} F(\alpha_0,\mu_0,\k_0)$. 
\end{enumerate}
\end{theorem}
\begin{proof}
The theorem was proved as \cite[Theorem~3.21]{banajifeliu26}. 
\end{proof}

It was shown in \cite{banajifeliu26} that when asking about existence and degeneracy of equilibria of a mass action network on the stoichiometric class $S_K$, we can use \Cref{thm:degensolvability} to transfer attention from the {\em original} mass action equation system $g(x,\k)=0,\,\,Zx=K$, to the {\em alternative equations} $f_K(\alpha,\mu, \k)=0$; see also \cite{regensburger:gale} for results in this direction. This allowed us, for example, to bound the number of positive nondegenerate equilibria on any stoichiometric class for several classes of networks. Here we will prove that when asking questions about the occurrence, nondegeneracy and transversality of fold and cusp bifurcations, we can likewise transfer attention to the alternative equations. 

\begin{remark}[Special cases of the alternative equations]
\label[remark]{remspecial}
We gather some special cases where some of the matrices in question are empty:
\begin{itemize}
\item When $m=s_\mP+p$, $W$ is empty, and the alternative equations become the system of $n-r$ equations in $n-r$ variables:
\[
f_K(\alpha, \mu,\k) := Z[(h(\alpha)/\k)^{\wG}\circ \mu^{\wQ}] - K=0\,.
\]
We remark that in the case that $W$ is empty, the sources of the network were referred to as {\bf $\mP$-independent} in \cite{banajifeliu26}. When $\mP$ is the trivial partition, $\mP$-independence is simply affine independence.

\item When $m= r+p$, $Y$ consists of a single point, say $\alpha_0$, and we may write the alternative equations as
\[
f_K(\mu,\k) := \left(\begin{array}{c}\hat{c}_0^W - \k^W\\Z[(\hat{c}_0/\k)^{\wG}\circ \mu^{\wQ}] - K\end{array}\right)=0\,.
\]
where $\hat{c}_0:=h(\alpha_0)$. In this case, either the network is $\mP$-overdetermined, in which case nondegenerate positive equilibria (hence fold and cusp bifurcations) are impossible, or $m=s_\mP+p$, in which case $W$ is empty and the alternative equation system becomes, simply,
\[
f_K(\mu,\k) := Z[(\hat{c}_0/\k)^{\wG}\circ \mu^{\wQ}] - K=0\,.
\]

\item When $n=s_\mP$, then the variables $\mu$ are empty, and $\wQ$ is an empty matrix: in this case, the alternative equations become
\[
f_K(\alpha,\k) := \left(\begin{array}{c}h(\alpha)^W - \k^W\\Z[(h(\alpha)/\k)^{\wG}] - K\end{array}\right)=0\,.
\]

\item When $n=r$, namely the network has full-rank and consequently $Z$ is empty, either the system is $\mP$-overdetermined, hence nondegenerate positive equilibria are impossible, or $n=s_\mP$ and hence the set of variables $\mu$ is empty, and $\wQ$ is empty. In this case, the alternative equations are simply the solvability equations, namely:
\[
f(\alpha,\k) := h(\alpha)^W - \k^W=0\,.
\]

\item When $m=s_\mP+p$ and $n=r$, so that both $W$ and $Z$ are empty, then as $\mC$ is nonempty, hence $m \geq r+p$, we must have $m=r+p$. In this case, the alternative system degenerates to an empty system on a domain consisting of a single point, which is trivially always satisfied. Fold and cusp bifurcations are clearly forbidden.
\end{itemize}
\end{remark}

To illustrate the construction of the alternative equations in the full-rank case, we revisit the network in Example 1. Reference \cite{banajifeliu26} includes numerous further examples of the construction of the alternative equations. 
\begin{ex}
\label[example]{Ex1a}
Recall the following $(2,4,2)$ network from \Cref{Ex1}:
\[
\fX \ce{->[\k_1]} \mathsf{0}\,,\quad \mathsf{0} \ce{->[\k_2]} \fX + \fY,\quad \fX + \fY \ce{->[\k_3]} 3\fX\,,\quad 2\fX \ce{->[\k_4]} 3\fX\,.
\]
In this case, using the trivial partition, we have
\[
\Gamma = \left(\begin{array}{rrrr}-1&1&2&1\\0&1&-1&0\end{array}\right), \quad [A\,|\,\bm{1}] = \left(\begin{array}{ccc}1&0&1\\0&0&1\\1&1&1\\2&0&1\end{array}\right)\,,\quad W=(2,-1,0,-1)\,.
\]
We may define $h(\alpha)=(2\alpha+1,\alpha,\alpha,1-\alpha)^\top$ with domain $(0,1)$. As $\ker[A\,|\,\bm{1}]$ is trivial, the alternative equation $h(\alpha)^W = \k^W$ takes the form $(2\alpha+1)^2/(\alpha(1-\alpha)) = \k_1^2/(\k_2\k_4)$, which, after clearing denominators (justified by \Cref{prop:prefactor}) and writing $\theta := \frac{\k_1^2}{\k_2\k_4}$, becomes
\[
(4+\theta)\alpha^2 + (4-\theta)\alpha + 1 = 0\,,
\]
as in \Cref{Ex1}.
\end{ex}

\section{Main results}
\label{secmain}

In all results to follow, to avoid trivialities, we assume that the alternative system of equations is nonempty. 

{\bf Notation and assumptions.} Throughout this section, we fix the following notation and assumptions. We fix an equilibrium $\hat{p}_0 := (x_0, \k_0) \in g^{-1}(0)$, and let $K = Zx_0$. We write $p_0:= (\alpha_0, \mu_0, \k_0)$ where, via \Cref{thm:degensolvability}(i), $(\alpha_0, \mu_0)$ is uniquely defined by $f_K(\alpha_0, \mu_0, \k_0) = 0$, $F(\alpha_0, \mu_0, \k_0) = x_0$. Write $c_0:= \k_0 \circ x_0^A$ and note that $c_0 \in \mC$. Write $\hat{c}_0 := h(\alpha_0)$, and note that $\hat{c}_0 \in \hat{\mC}$, and in particular $c_0 = (\bm{1}_{\mP}\lambda_0) \circ \hat{c}_0$ for some uniquely defined $\lambda_0 \in \R^p_+$. We write $\wJ:=g_x(\hat{p}_0)$, and $J:=\partial_{(\alpha, \mu)}f_K(p_0)$. We assume that $\ker J$, equivalently by \Cref{thm:degensolvability}(ii), $\ker \wJ \cap \mathrm{im}\,\Gamma$, is one-dimensional. As usual, we let $\dot x = \tilde{g}_K(x, \k)$ denote the restriction of \eqref{eqODE} to $S_K$ (recall that $\tilde{g}_K\colon S_K \times \R^m_+ \to \mathrm{im}\,\Gamma$). 

The main result of this paper is the following. 

\begin{theorem}
\label[theorem]{foldcuspmain}
Fix $i,j \in \{1, \ldots, m\}$. Each condition B2--B5 holds for $\tilde{g}_K$ at $\hat{p}_0$ if and only if the corresponding condition holds for $f_K$ at $p_0$. Consequently: 
\begin{enumerate}
\item If $\hat{p}_0$ is a simply degenerate point of $\tilde{g}_K$, then the mass action system \eqref{eqODE} has a nondegenerate fold (resp., cusp) bifurcation at $\hat{p}_0$, and this bifurcation is unfolded transversally by $\k_i$ (resp., $\{\k_i, \k_j\}$), if and only if $f_K$ satisfies conditions B2$\,^{c}$ and B3 (resp., B2, B4 and B5) at $p_0$. 
\item If $\hat{p}_0$ is a simply degenerate point of $\tilde{g}_K$ and $p_0$ is a simply degenerate point of $f_K$, then the following are equivalent:
\begin{itemize}
\item The mass action system \eqref{eqODE} has a nondegenerate fold (resp., cusp) bifurcation at $\hat{p}_0$, and this bifurcation is unfolded transversally by $\k_i$ (resp., $\{\k_i, \k_j\}$).
\item $f_K$ has a nondegenerate fold (resp., cusp) bifurcation at $p_0$, and this bifurcation is unfolded transversally by $\k_i$ (resp., $\{\k_i, \k_j\}$).
\end{itemize}
\end{enumerate}
\end{theorem}
\begin{proof}
The result will follow immediately from \Cref{lemB2,lemB3,lemB4,lemB5} below, which cover conditions B2--B5 in turn. The two consequences are immediate. 
\end{proof}

\begin{remark}[The assumption of simple degeneracy]
\Cref{Ex1} above illustrates the need to introduce explicitly in \Cref{foldcuspmain}(2) the assumption that the relevant points are simply degenerate points of $\tilde{g}_K$ and $f_K$. In the original mass-action equations, when $\k_1^2 = 12 \k_2\k_4$ and $\k_3 = \frac{2}{3}\k_4$, the system has degenerate equilibria at which the Jacobian matrix has a zero eigenvalue of algebraic multiplicity two and, in fact, a Bogdanov--Takens bifurcation occurs \cite[Example 2]{BBH2024smallbif}. The alternative equations, consisting of one equation in a single variable cannot, of course, admit an eigenvalue of algebraic multiplicity $2$ or Bogdanov--Takens bifurcations, and thus fail to capture the full local bifurcation behaviour of the mass action network at these points. 
\end{remark}

\begin{remark}[Application of \Cref{foldcuspmain}]
As seen in the examples above, given a mass action network, we may apply \Cref{foldcuspmain} in two ways.
\begin{enumerate}
\item We first identify simply degenerate points of the mass action equations. We then use the alternative equations to check when the corresponding points satisfy conditions B2$^{c}$ and B3 (resp., B2, B4 and B5), giving us the fold (resp., cusp) bifurcation set. 

\item We show that the alternative equations always fail some conditions necessary for fold (resp., cusp) bifurcations, thus ruling out fold (resp., cusp) bifurcations in the mass action systems associated with some class of networks. 

\end{enumerate}

\end{remark}

The remainder of this section is devoted to the proof of \Cref{foldcuspmain}: we prove some preliminary results, and then the equivalence of conditions B2--B5 for the original and alternative systems.

\subsection{Higher derivatives of $f$ and $g$}

${}$
Recall our assumption that $\ker J$ is one-dimensional. Fix $0 \neq q =(q_1, q_2) \in \ker J$, define $\hat{q}:= J^{(F)}q$ and note that, by \Cref{thm:degensolvability}(ii), $\{\hat{q}\}$ spans $\ker \wJ \cap \mathrm{im}\,\Gamma$. 

{\bf Notation.} The following notation greatly shortens various expressions. Given $x \in \R^n$, define $x^* = x/x_0$ and $\hat{v}_x := Ax^* = A(x/x_0)$. For brevity let $\hat{v} := \hat{v}_{\hat{q}} = A(\hat{q}/x_0)$. Given $y = (y_1, y_2) \in \R^{m-r-p} \times \R^{n-s_{\mP}}$, define $v_y := D_{1/\hat{c}_0}H y_1$, and $w_y := y_2/\mu_0$. For brevity, let $v:= v_q = D_{1/\hat{c}_0}H q_1$, and $w := w_q = q_2/\mu_0$. Note that with this notation, using \eqref{eqJF},
\begin{equation}
\label{eqhatq}
\hat{q}:= J^{(F)}q = D_{x_0}(\wG v + \wQ w)\,.
\end{equation}

{\bf Derivatives of $g$.} Write $\wB(\cdot, \cdot)$ for $g_{xx}(\cdot, \cdot)$ evaluated at $\hat{p}_0$, and $\wC(\cdot, \cdot, \cdot)$ for $g_{xxx}(\cdot, \cdot, \cdot)$ evaluated at $\hat{p}_0$. We compute, for any $\hat{h} \in \R^n$,
\begin{equation}
\label{eqwbhat1}
\wB(\hat{q},\hat{h}) = \left.\frac{\partial}{\partial t}\left[\Gamma D_{\k\circ(x_0 + t\hat{q})^A} A D_{1/(x_0 + t\hat{q})} \hat{h}\right]\right|_{t=0} =  \Gamma D_{c_0} \left[\hat{v}\circ \hat{v}_{\hat{h}} - A(\hat{q}^*\circ \hat{h}^*)\right]\,.
\end{equation}
In particular,
\begin{equation}
\label{eqwbhat}
\wB(\hat{q},\hat{q}) = \Gamma D_{c_0}\left[\hat{v}^2  -A(\hat{q}^*)^2\right]\,.
\end{equation}
We also calculate
\begin{eqnarray}
\label{eqwchat}
\wC (\hat{q},\hat{q},\hat{q}):= \left.\frac{\partial^3}{\partial t^3}g(x_0 + t\hat{q}, \k_0)\right|_{t=0} 
&=& \Gamma D_{c_0}\left[\hat{v}^3 - 3D_{\hat{v}}A (\hat{q}^*)^2 + 2A(\hat{q}^*)^3\right]\,.
\end{eqnarray}

{\bf Derivatives of $f$.} Let $y := (\alpha, \mu)$. Write $B(\cdot, \cdot)$ for $f_{yy}(\cdot, \cdot)$ evaluated at $p_0$ and $C(\cdot, \cdot, \cdot)$ for $f_{yyy}(\cdot, \cdot, \cdot)$ evaluated at $p_0$. Then
\begin{eqnarray}
\label{eqwb1}
B(q,h) &:=& \left.\frac{\partial}{\partial t}\left[\frac{\partial}{\partial t'} f_K(\alpha_0 + t'h_1 + tq_1, \mu_0 + t'h_2 + tq_2, \k_0)\right]\right|_{t=t'=0}\\ \nonumber 
&=& \left(\begin{array}{c}-D_{\k_0^W} W (v\circ v_h)\\Z D_{x_0} [\hat{q}^*\circ(\wG v_h + \wQ w_h) - \wG (v \circ v_h) - \wQ (w\circ w_h)]\end{array}\right)\,.
\end{eqnarray}
In particular,
\begin{equation}
\label{eqwb}
B(q,q) = \left(\begin{array}{c}-D_{\k_0^W}W v^2\\Z D_{x_0}[(\hat{q}^*)^2 - \wG v^2 -\wQ w^2]\end{array}\right)
\end{equation}

We calculate
\begin{eqnarray}
\label{eqwc}
C(q,q,q) &:=& \left.\frac{\partial^3}{\partial t^3} f_K(\alpha_0 + tq_1, \mu_0 + tq_2, \k_0)\right|_{t=0}\\\nonumber &=& 
\left(\begin{array}{c} 2 D_{\k_0^W}W v^3\\Z D_{x_0}[(\hat{q}^*)^3 - 3(\hat{q}^*)\circ (\wG v^2 + \wQ w^2) + 2\wG v^3 +2\wQ w^3]\end{array}\right)\,.
\end{eqnarray}

\subsection{Preliminary lemmas} 

Define the following linear subspaces of $\R^m$. 
\[
\mA:= \mathrm{im}(A D_{1/x_0}\Gamma), \quad \mY:= \mathrm{im}(D_{1/\hat{c}_0}H), \quad \mZ:= \mathrm{im}(\bm{1}_\mP)\,, \quad \mM = \mA + \mY + \mZ\,.
\]
Our strategy will be to show that various bifurcation conditions, nondegeneracy conditions, and transversality conditions are equivalent to certain quantities belonging to $\mM$. 

The following lemma gathers results about the subspaces $\mA, \mY, \mZ$ and their sum $\mM$. 
\begin{lemma}
\label[lemma]{lemM1M2M3}
(a) $\mY \cap \mZ = \{0\}$. (b) $\ker \Gamma = D_{\hat{c}_0}(\mY+\mZ)$.
(c) $\mZ^k = \mZ$ for all $k\geq 1$ and $\mZ\mY = \mY$. (d) $\mA \cap (\mY+\mZ) = \mathrm{span}\{\hat{v}\}$. (e) $\hat{v} - v \in \mZ$. (f) $\mY \cap (\mA+\mZ) = \mathrm{span}\{v\}$. (g) Given $v_1, v_2, v_3 \in \R^m$, if $v_1 \in (v_3 + \mA + \mZ) \cap \mY$ and $v_2 \in (v_3 + \mY + \mZ) \cap \mA$, then $v_1 + v_2 \in v_3 + \mathrm{span}\{\hat{v}\} + \mZ$. (h) $\hat{v}^2 -v^2 \in \mY + \mZ$, equivalently $D_{\hat{c}_0}(\hat{v}^2 - v^2) \in \ker \Gamma$. (i) $\mathrm{dim}\,\mM = m-1$. 
\end{lemma}
\begin{proof}
(a) and (b). See \cite[Lemma 3.5]{banajifeliu26}.

(c) $\mZ^k = \mZ$ is obvious. To see that $\mZ\mY = \mY$, note that given $z \in \mZ$ and $y = D_{1/\hat{c}_0}H s \in \mY$, we have
\[
z\circ y = D_{z}D_{1/\hat{c}_0}H s = D_{1/\hat{c}_0}[z\circ (H s)] = D_{1/\hat{c}_0}H s' \in \mY
\]
for some $s'$, as $\mathrm{im}\,H$ factors over $\mP$, see \cite[Lemma 3.5]{banajifeliu26}.

(d) First, $\hat{v} \in \mA \cap (\mY+\mZ)$: by definition $\hat{v}:= AD_{1/x_0}\hat{q}$, and since $\hat{q} \in \mathrm{im}\,\Gamma$, it follows that $\hat{v} \in \mA$; on the other hand, 
\[
0 = \wJ\hat{q} = \Gamma (c_0\circ \hat{v}) = \Gamma (z\circ \hat{c}_0\circ \hat{v})
\]
for some $z \in \mZ$ from which (as $\ker\Gamma$ factors over $\mP$) it follows that $\hat{c}_0 \circ \hat{v} \in \ker \Gamma$, equivalently (using part (b)) $\hat{v} \in \mY + \mZ$. Given $0 \neq x \in \mA \cap (\mY+\mZ)$, we have $x = AD_{1/x_0}\Gamma s$ for some $s$, and $\Gamma D_{\hat{c}_0} x = 0$ (by part (b)). Putting these together gives $\Gamma D_{\hat{c}_0}AD_{1/x_0}\Gamma s= 0$, namely (as $\ker \Gamma$ factors over $\mP$), $\wJ\Gamma s = \Gamma D_{c_0}AD_{1/x_0}\Gamma s = 0$ from which it follows that $\Gamma s \in \mathrm{span}\{\hat{q}\}$, implying, $x\in\mathrm{span}\{\hat{v}\}$.

(e) The first component of $Jq=0$ tells us that $v \in \ker W =  \mathrm{im}\,[A\,|\,\bm{1}_\mP]$. On the other hand,
\[
\hat{v}:= A\hat{q}^* = A(\wG v + \wQ w)\,.
\]
Using \Cref{lemGsolve}(i), we have $\hat{v}-v = A(\wG v + \wQ w) - v \in \mZ$.

(f) First, $v \in \mY \cap (\mA + \mZ)$: by definition $v := D_{1/\hat{c}_0}H q_1 \in \mY$; meanwhile, from (e), $v \in \hat{v} + \mZ \subseteq \mA + \mZ$. Now, given $0 \neq y \in \mY \cap (\mA+\mZ)$, we have $\Gamma D_{\hat{c}_0} y = 0$ (as $y \in \mY$). On the other hand, as $y \in \mA + \mZ$, we can write $y = A D_{1/x_0} \Gamma s + z$, for some $s \in \R^m, z \in \mZ$, namely, $\Gamma D_{\hat{c}_0} A D_{1/x_0} \Gamma s + \Gamma D_{\hat{c}_0} z = 0$. The second term is automatically zero since $\hat{c}_0 \in \ker \Gamma$, hence $\hat{c}_0 \circ z \in \ker \Gamma$, so $\Gamma D_{\hat{c}_0} A D_{1/x_0} \Gamma s = 0$, namely (as $\ker \Gamma$ factors over $\mP$), $\wJ\Gamma s = \Gamma D_{c_0} A D_{1/x_0} \Gamma s = 0$ from which it follows that $\Gamma s = t \hat{q}$ for some $t \in \R$. So, using (e), $y = t\hat{v} + z = t v + z'$, for some $z' \in \mZ$. But we must have $z'=0$ since $y,v \in \mY$ and, by part (a), $\mY \cap \mZ = \{0\}$. Thus $y \in \mathrm{span}\,\{v\}$.

(g) Write $v_1 = v_3 + a + z = y$ and $v_2 = v_3 + y' + z' = a'$, where $a,a' \in \mA$, $y,y' \in \mY$ and $z,z' \in \mZ$. Then $v_3 = y-a-z = a'-y'-z'$, from which $a + a' = y + y' -z + z' = t \hat{v}$ for some $t \in \R$ by part (d). Thus $v_1 + v_2 = v_3 + a'+ a + z = v_3 + t\hat{v} + z$.

(h) Note that $v \in \mY$ and, by part (e), $\hat{v} = v + z$ for some $z \in \mZ$. Consequently, by part (c), 
\[
\hat{v}^2 -v^2 =  2z \circ v + z^2 \in \mY + \mZ\,,
\]
equivalently, by part (b), $D_{\hat{c}_0}(\hat{v}^2 - v^2) \in \ker \Gamma$.

(i) We show that $D_{\hat{c}_0}\mM = \mathrm{im}\,(D_{\hat{c}_0}AD_{1/x_0}\Gamma) + \ker \Gamma$ has dimension $m-1$ from which the claim follows. By assumption, $\mathrm{rank}\,\wJ\Gamma = r-1$. A vector $s \in \mathrm{im}\,(D_{\hat{c}_0}AD_{1/x_0}\Gamma) \cap \ker \Gamma$ satisfies, for some $s' \in \R^m$, $s = D_{\hat{c}_0}AD_{1/x_0}\Gamma s'$ and $\Gamma s = 0$, giving $\wJ\Gamma s'=0$ (we used that $\ker \Gamma$ factors over $\mP$). But then, $\Gamma s' = t\hat{q}$ for some $t \in \R$. If $\hat{v} := AD_{1/x_0} \hat{q} = 0$, then $s=0$, i.e., the intersection $\mathrm{im}\,(D_{\hat{c}_0}AD_{1/x_0}\Gamma) \cap \ker \Gamma$ is trivial. In this case, $r-1 = \mathrm{rank}\,\wJ\Gamma \leq \mathrm{rank}(D_{\hat{c}_0}AD_{1/x_0}\Gamma) \leq r-1$, and $\mathrm{im}\,(D_{\hat{c}_0}AD_{1/x_0}\Gamma) + \ker \Gamma$ is the direct sum of an $(r-1)$-dimensional and an $(m-r)$-dimensional subspace proving the result in this case. If $\hat{v} \neq 0$, then $\mathrm{rank}(D_{\hat{c}_0}AD_{1/x_0}\Gamma) = r$, and the intersection $\mathrm{im}\,(D_{\hat{c}_0}AD_{1/x_0}\Gamma) \cap \ker \Gamma = \mathrm{span}\{\hat{c}_0 \circ \hat{v}\}$; so $\mathrm{im}\,(D_{\hat{c}_0}AD_{1/x_0}\Gamma) + \ker \Gamma$ is the sum of an $r$-dimensional subspace and an $(m-r)$-dimensional subspace with 1D intersection; again, this proves the result.
\end{proof}

The next lemma will be used frequently to translate conditions on membership of $\mathrm{im}\,J$ and $\mathrm{im}\,(\wJ \Gamma)$ into conditions involving membership of $\mM$. 
\begin{lemma}
\label[lemma]{lembasic}
Given vectors $u_0$, $u_1$, $u_2$, $u_3$ of appropriate dimension, we have the following equivalences:
\begin{enumerate}
\item[(a)] $\Gamma D_{c_0} u_0 \in \mathrm{im}\,(\wJ\Gamma) \quad \Leftrightarrow \quad u_0 \in \mM$. \vskip 2mm
\item[(b)] $\displaystyle{\left(\begin{array}{c}D_{\k_0^W}Wu_1\\ZD_{x_0}[\wG u_1 + \wQ u_2 + u_3]\end{array}\right) \in \mathrm{im}\,J\quad \Leftrightarrow \quad u_1 + Au_3 \in \mM}$. 
\end{enumerate}
\end{lemma}
\begin{proof}
(a) The condition on the left is clearly equivalent to the existence of $a \in \mA$ s.t. $u_0-a \in D_{1/\hat{c}_0}\ker \Gamma = \mY + \mZ$, where we have used \Cref{lemM1M2M3}(b) and the fact that $\ker \Gamma$ factors over $\mP$. The claim follows. 

(b) Using \eqref{eqJf}, the condition on the left is equivalent to the existence of $s_1, s_2$ s.t. the pair of equations
\[
D_{\k_0^W}W(u_1-D_{1/\hat{c}_0}H s_1) = 0, \quad ZD_{x_0}[\wG (u_1 - D_{1/\hat{c}_0}H s_1) + \wQ (u_2 - s_2/\mu_0) + u_3] = 0
\]
is satisfied. The first equation is equivalent to $u_1-D_{1/\hat{c}_0}H s_1 \in \mathrm{im}\,[A\,|\,\bm{1}_\mP]$; the second is equivalent to the existence of $s_3$ s.t. $\wG (u_1 - D_{1/\hat{c}_0}H s_1) + \wQ (u_2 - s_2/\mu_0) + u_3 =D_{1/x_0}\Gamma s_3$. For any fixed $s_1, s_3$, \Cref{lemGsolve} tells us that we can choose $s_2$ to satisfy this latter equation if and only if there exists $z \in \mZ$ s.t. $u_1 - D_{1/\hat{c}_0}H s_1 + Au_3 = AD_{1/x_0}\Gamma s_3 + z$. The pair of equations are thus equivalent to $u_1 + Au_3 \in \mM$. (Note that this certainly also implies the first condition $u_1-D_{1/\hat{c}_0}H s_1 \in \mathrm{im}\,[A\,|\,\bm{1}_\mP]$.)
\end{proof}

The following technical lemma is useful. Given $x \in \R^n$, recall the notation $x^*:= x/x_0$, $\hat{v}_x := A(x/x_0)$. 
\begin{lemma}
\label[lemma]{lemtransfer}
Let $x \in \mathrm{im}\,\Gamma$, $u_0 \in \mY$, $u_0+u_1 \in \mathrm{im}\,[A\,|\,\bm{1}_\mP]$, and suppose that $u_0, u_1, u_2, u_3$ satisfy the simultaneous equations
\begin{equation}
\label{eqtech}
\Gamma D_{c_0}[\hat{v}_x - u_1 - Au_3]=0\,, \quad ZD_{x_0}[\wG(u_0+u_1) + \wQ u_2 + u_3]=0\,.
\end{equation}
Then there exist $t \in \R$ and $z \in \mZ$ s.t. (i) $\wG(u_0+u_1) + \wQ u_2 + u_3 = x^* + t \hat{q}^*$ and (ii) $u_0 = \hat{v}_x -u_1 - Au_3 + t\hat{v} + z$. 
\end{lemma}
\begin{proof}
For brevity, let $\theta:= u_1 + Au_3$ and $y:= \wG(u_0+u_1) + \wQ u_2 + u_3$. Multiplying the latter definition by $A$ gives
\begin{equation}
\label{eqA1}
Ay = u_0 + \theta + z\,,
\end{equation}
for some $z \in \mZ$. Observe that the second equation of \eqref{eqtech} implies that $Ay \in \mA$. Thus $-u_0 \in (\theta + \mA + \mZ) \cap \mY$. Using the first equation of \eqref{eqtech}, $\hat{v}_x \in (\theta + \mY+\mZ) \cap \mA$. By \Cref{lemM1M2M3}(g), the previous two observations imply that there exists $\tilde{t} \in \R$ and $z' \in \mZ$ s.t. 
\begin{equation}
\label{eqA2}
\hat{v}_x - u_0 = \theta + \tilde{t} \hat{v} + z'\,.
\end{equation}
Frome \eqref{eqA1} and \eqref{eqA2} we get $Ay = \hat{v}_x - \tilde{t}\hat{v} + z''$ for some $z'' \in \mZ$. But then 
\[
\wJ D_{x_0}y = \Gamma D_{c_0}Ay = \Gamma D_{c_0}[\hat{v}_x - \tilde{t}\hat{v} + z''] = \wJ x\,.
\]
(We have used the fact that $\Gamma D_{c_0}\hat{v}$ and $\Gamma D_{c_0}z''$ are automatically zero.) Thus, since $D_{x_0}y \in \mathrm{im}\,\Gamma$ by the second equation of \eqref{eqtech}, $D_{x_0} y = x + t \hat{q}$ for some $t \in \R$. Multiplying through by $D_{1/x_0}$ gives the first conclusion. Multiplying through by $AD_{1/x_0}$ and using \eqref{eqA1} gives
\[
A y = \hat{v}_x + t \hat{v} = u_0 + \theta + z\,,
\]
leading to the second conclusion. 
\end{proof}

\subsection{The main lemmas}
We note first that $\mathrm{im}\,J^{(\tilde{g}_K)} = \mathrm{im}\,(\wJ\Gamma)$. 

\begin{lemma}[condition B2]
\label[lemma]{lemB2}
The following are equivalent: 
\begin{enumerate}
\item $B(q,q) \in \mathrm{im}\,J$,
\item $\wB (\hat{q}, \hat{q}) \in \mathrm{im}\,(\wJ\Gamma)$,
\item $\hat{v}^2 - A(\hat{q}^*)^2 \in \mM$,
\item $v^2 - A(\hat{q}^*)^2 \in \mM$,
\item $D_{\hat{c}_0}[\hat{v}^2 - A(\hat{q}^*)^2] \in \mathrm{im}\,(D_{\hat{c}_0}AD_{1/x_0}\Gamma) + \ker \Gamma\,.$
\end{enumerate}
Consequently, condition B2 holds for $\tilde{g}_K$ at $\hat{p}_0:=(x_0, \k_0)$ if and only if it holds for $f_K$ at $p_0:=(\alpha_0, \mu_0, \k_0)$. 
\end{lemma}

\begin{proof}
(3) $\Leftrightarrow$ (4) follows from \Cref{lemM1M2M3}(h) which tells us that $\hat{v}^2 - v^2 \in \mY + \mZ \subseteq \mM$, while (3) $\Leftrightarrow$ (5) follows from \Cref{lemM1M2M3}(b) as $D_{\hat{c}_0}(\mY + \mZ) = \ker \Gamma$.

(2) $\Leftrightarrow$ (3). Using Equation~\eqref{eqwbhat} and \Cref{lembasic}(a), $\wB (\hat{q},\hat{q}) = \Gamma D_{c_0} [\hat{v}^2 - A (\hat{q}^*)^2] \in \mathrm{im} (\wJ\Gamma)$ if and only if $\hat{v}^2 - A (\hat{q}^*)^2  \in \mM$.

(1) $\Leftrightarrow$ (4). Using Equation~\eqref{eqwb} and \Cref{lembasic}(b), $B(q,q) \in \mathrm{im}\,J$ if and only if $v^2 - A(\hat{q}^*)^2  \in \mM$.
\end{proof}

\vskip 5mm
\begin{lemma}[condition B3]
\label[lemma]{lemB3}
For any fixed $i \in \{1, \ldots, m\}$, the conditions $\partial_{\k_i} g \in\mathrm{im}\,(\wJ \Gamma)$ and $\partial_{\k_i}f_K \in \mathrm{im}\,J$ are equivalent. Hence, condition B3 holds for $\tilde{g}_K$ at $\hat{p}_0:=(x_0, \k_0)$ if and only if it holds for $f_K$ at $p_0:=(\alpha_0, \mu_0, \k_0)$. 
\end{lemma}
\begin{proof}
Fix $i \in \{1, \ldots, m\}$. We have, using \eqref{eqJstarg} and \eqref{eqJfstar},
\[
\partial_{\k_i} \tilde{g}_K = \partial_{\k_i} g = \Gamma D_{c_0} (e_i/\k_0), \quad \partial_{\k_i}f_K = \left(\begin{array}{c}-D_{\k_0^W}W\\-ZD_{x_0}\wG\end{array}\right)(e_i/\k_0)\,.
\]
From \Cref{lembasic} both $\partial_{\k_i} g \in\mathrm{im}\,(\wJ \Gamma)$ and $\partial_{\k_i}f_K \in \mathrm{im}\,J$ are equivalent to $e_i/\k_0 \in \mM$, and hence to each other. 
\end{proof}

From now on, we assume that $B(q,q) \in \mathrm{im}\,J$, equivalently, by \Cref{lemB2}, $\wB(\hat{q},\hat{q}) \in \mathrm{im}\,(\wJ \Gamma)$. Recall that nondegeneracy of a potential cusp point of the original system is equivalent to 
\[
\wC(\hat{q}, \hat{q}, \hat{q}) - 3 \wB(\hat{q}, \hat{h}) \not \in \mathrm{im}\,(\wJ \Gamma)\,,
\]
where $\hat{h}$ is any vector in $\wJ^{-1}(\wB(\hat{q},\hat{q}))\cap \mathrm{im}\,\Gamma$. Meanwhile, nondegeneracy of a potential cusp point of the alternative system is equivalent to 
\[
C(q,q,q) - 3B(q,h) \not \in \mathrm{im}\,J\,,
\]
where $h$ is any vector in $J^{-1}(B(q,q))$. Our next lemma shows that these two conditions are equivalent. 

\begin{lemma}[condition B4]
\label[lemma]{lemB4}
Let $\hat{h} \in \wJ^{-1}(\wB(\hat{q},\hat{q}))\cap \mathrm{im}\,\Gamma$ and $h \in J^{-1}(B(q,q))$. The following are equivalent:
\begin{enumerate}
\item $\wC(\hat{q}, \hat{q}, \hat{q}) - 3 \wB(\hat{q}, \hat{h})  \in \mathrm{im}(\wJ\Gamma)$,
\item $C(q,q,q) - 3B(q,h) \in \mathrm{im}\,J$,
\item $\widehat{\Phi}:= 2(\hat{v}^3 - A(\hat{q}^*)^3) + 3\hat{v}\circ (\hat{v}_{\hat{h}} - \hat{v}^2 + A (\hat{q}^*)^2) -3A(\hat{q}^*\circ \hat{h}^*) \in \mM$.
\end{enumerate}
Consequently, condition B4 holds for $\tilde{g}_K$ at $\hat{p}_0:=(x_0, \k_0)$ if and only if it holds for $f_K$ at $p_0:=(\alpha_0, \mu_0, \k_0)$. 
\end{lemma}

\begin{proof}
Using \eqref{eqwbhat}, the equation $\wJ \hat{h} = \wB(\hat{q},\hat{q})$ tells us that $\Gamma D_{c_0}[\hat{v}_{\hat{h}} - \hat{v}^2 + A(\hat{q}^*)^2] = 0$, equivalently, using \Cref{lemM1M2M3}(b)~and~(h),
\begin{equation}
\label{eqcstar1}
\Gamma D_{c_0}[\hat{v}_{\hat{h}} - v^2 + A(\hat{q}^*)^2] = 0\,.
\end{equation}
Consequently, $\hat{v}_{\hat{h}} - \hat{v}^2 + A(\hat{q}^*)^2 \in \mY + \mZ$. Note, further, that $\hat{h} \in \mathrm{im}\,\Gamma$ implies that $\hat{h}^* \in \mathrm{im}(D_{1/x_0}\Gamma)$ and $\hat{v}_{\hat{h}} \in \mA$. 

On the other hand, expanding $Jh = B(q,q)$ implies that $v_h+v^2 \in \mathrm{im}[A\,|\,\bm{1}_\mP]$, and 
\begin{equation}
\label{eqcstar2}
ZD_{x_0}y^* = 0, \quad \mbox{where} \quad y^*:=(\wG (v_h+v^2) + \wQ (w_h+w^2) - (\hat{q}^*)^2)\,,
\end{equation}
hence $y^* \in \mathrm{im}D_{1/x_0}\Gamma$.  

(1) $\Leftrightarrow$ (3). Recalling equations \eqref{eqwchat} and \eqref{eqwbhat1}, 
\[
\wC(\hat{q}, \hat{q}, \hat{q}) - 3 \wB(\hat{q}, \hat{h}) = \Gamma D_{c_0}\left[\hat{v}^3 - 3\hat{v}\circ (A (\hat{q}^*)^2 + \hat{v}_{\hat{h}}) + 2A(\hat{q}^*)^3  +3A(\hat{q}^*\circ \hat{h}^*)\right]\,.
\]
Thus, by \Cref{lembasic}(a), the condition $\wC(\hat{q}, \hat{q}, \hat{q}) - 3 \wB(\hat{q}, \hat{h}) \in \mathrm{im}(\wJ\Gamma)$ is equivalent to
\[
\widehat{\Phi}:= 2(\hat{v}^3 - A(\hat{q}^*)^3) + 3\hat{v}\circ (\hat{v}_{\hat{h}} - \hat{v}^2+A (\hat{q}^*)^2) -3A(\hat{q}^*\circ \hat{h}^*) \in \mM
\]

(2) $\Leftrightarrow$ (3). Recalling Equations \eqref{eqwc} and \eqref{eqwb1},
\[
C(q,q,q) - 3B(q,h) = \left(\begin{array}{c} D_{\k_0^W}W (2v^3 +3 v\circ v_h)\\Z D_{x_0}[\wG(2 v^3 + 3(v \circ v_h)) +\wQ(2w^3 +3(w\circ w_h)) -2(\hat{q}^*)^3 - 3\hat{q}^* \circ y^*] \end{array}\right)\,.
\]
By \Cref{lembasic}(b), $C(q,q,q) - 3B(q,h) \in \mathrm{im}\,J$ if and only if 
\[
\Phi' := 2 v^3 + 3(v \circ v_h) - 2A(\hat{q}^*)^3 - 3A(\hat{q}^* \circ y^*) \in \mM\,.
\]
By \Cref{lemtransfer} applied to \eqref{eqcstar1}~and~\eqref{eqcstar2}, $y^* = \hat{h}^* + t\hat{q}^*$, and $v_h = -v^2 + A(\hat{q}^*)^2 + \hat{v}_{\hat{h}} +t\hat{v} + z$ for some $t \in \R$, $z \in \mZ$. Substituting into the expression for $\Phi'$, and observing that $v^2 -A(\hat{q}^*)^2 \in \mM$ (\Cref{lemB2}), and $v \circ z \in \mM$ (\Cref{lemM1M2M3}(c)) gives that $\Phi' \in \mM$ if and only if 
\[
\Phi:= -v^3 -[2A(\hat{q}^*)^3 - 3v \circ (A(\hat{q}^*)^2)] + 3[v\circ\hat{v}_{\hat{h}} - A(\hat{q}^* \circ \hat{h}^*)] \in \mM\,.
\]
It remains to prove that $\widehat{\Phi} - \Phi \in \mM$. The reader may check that
\[
\widehat{\Phi}- \Phi = 3(\hat{v}_{\hat{h}} - \hat{v}^2+A (\hat{q}^*)^2)\circ (\hat{v}-v) +3v\circ(\hat{v}-v)^2 + 2(\hat{v}-v)^3\,,
\]
which is in $\mM$ as $\hat{v}-v \in \mZ$ (\Cref{lemM1M2M3}(e)), $v \in \mY$, and $\hat{v}_{\hat{h}} - \hat{v}^2+A (\hat{q}^*)^2 \in \mY+\mZ$ from above.
We have also used \Cref{lemM1M2M3}(c). This completes the proof. 
\end{proof}

Finally, we show that Condition B5 holds for $\tilde{g}_K$ at $\hat{p}_0:=(x_0, \k_0)$ if and only if it holds for $f_K$ at $p_0:=(\alpha_0, \mu_0, \k_0)$. Fix $i,j \in \{1, \ldots, m\}$ with $i \neq j$, and $(K_1, K_2) \neq (0,0)$. Define the operator $\mD:= K_1 \partial_{\k_i} + K_2 \partial_{\k_j}$, and the vector $\theta:=K_1 e_i/\k_0 + K_2 e_j/\k_0$. Note that $\mD \tilde{g}_K = \mD g$. 
\begin{lemma}[condition B5]
\label[lemma]{lemB5}
(i) $\mD g \in \mathrm{im}\,(\wJ \Gamma)$ if and only if $\mD f_K \in \mathrm{im}\,J$, with both being equivalent to $\theta \in \mM$. (ii) Suppose that $\mD g \in \mathrm{im}\,(\wJ \Gamma)$, equivalently, $\mD f_K \in \mathrm{im}\,J$. Then $\mD[\partial_x g]\hat{q} - \wB(\hat{q},\hat{h}) \in \mathrm{im}\,(\wJ \Gamma)$ if and only if $\mD[\partial_{(\alpha, \mu)} f]q  - B(q,h) \in \mathrm{im}\,J$, where $\hat{h}$ is any element of $\wJ^{-1}(\mD g) \cap \mathrm{im}\,\Gamma$, and $h$ is any element of $J^{-1}(\mD f)$. Consequently, condition B5 holds for $\tilde{g}_K$ at $\hat{p}_0:=(x_0, \k_0)$ if and only if it holds for $f_K$ at $p_0:=(\alpha_0, \mu_0, \k_0)$. 
\end{lemma}
\begin{proof}
First, we have, using \eqref{eqJstarg} and \eqref{eqJfstar},
\[
\mD g = \Gamma D_{c_0}\theta, \quad \mD f_K=\left(\begin{array}{c}-D_{\k_0^W}W\\-ZD_{x_0}\wG \end{array}\right) \theta\,.
\]
By \Cref{lembasic}, both $\mD g \in \mathrm{im}\,(\wJ \Gamma)$ and $\mD f_K \in \mathrm{im}\,J$ are equivalent to $\theta \in \mM$, proving (i). Let us now assume this to be the case. To prove (ii), we show that both $\mD[\partial_x g]\hat{q} - \wB(\hat{q},\hat{h}) \in \mathrm{im}\,(\wJ \Gamma)$ and $\mD[\partial_{(\alpha, \mu)} f_K]q  - B(q,h) \in \mathrm{im}\,J$ are equivalent to the condition $\hat{v} \circ \theta - \hat{v} \circ \hat{v}_{\hat{h}} + A(\hat{q}^* \circ \hat{h}^*) \in \mM$.

We easily calculate that $\mD[\partial_x g]\hat{q} = \Gamma D_{c_0}(\hat{v} \circ \theta)$; hence, using \eqref{eqwbhat1}, 
\[
\mD[\partial_x g]\hat{q} - \wB(\hat{q},\hat{h}) = \Gamma D_{c_0}(\hat{v} \circ \theta - \hat{v} \circ \hat{v}_{\hat{h}} + A(\hat{q}^* \circ \hat{h}^*))\,.
\]
By \Cref{lembasic}(a), the condition $\mD[\partial_x g]\hat{q} - \wB(\hat{q},\hat{h}) \in \mathrm{im}\,(\wJ \Gamma)$ is equivalent to
\begin{equation}
\label{eqB5a}
\Psi:=\hat{v} \circ \theta - \hat{v} \circ \hat{v}_{\hat{h}} + A(\hat{q}^* \circ \hat{h}^*) \in \mM\,.
\end{equation}
On the other hand, using \eqref{eqJf}~and~\eqref{eqhatq},
\[
\mD[\partial_{(\alpha, \mu)}f_K]q = \mD \left(\begin{array}{c}D_{\hat{c}_0^W}Wv \\ZD_{(\hat{c}_0/\k)^{\wG}\circ \mu_0^{\wQ}}\hat{q}^*\end{array}\right) = \left(\begin{array}{c}0 \\-ZD_{x_0}(\hat{q}^* \circ (\wG \theta))\end{array}\right)\,.
\]
Additionally, using \eqref{eqwb1}, 
\[
\mD[\partial_{(\alpha, \mu)} f_K]q  - B(q,h) = \left(\begin{array}{c}D_{\k_0^W} W (v\circ v_{h})\\Z D_{x_0} [\wG (v \circ v_{h}) + \wQ (w\circ w_{h}) - \hat{q}^*\circ(\wG (v_{h}+\theta) + \wQ w_{h})]\end{array}\right)\,.
\]
By \Cref{lembasic}(b), the condition $\mD[\partial_{(\alpha, \mu)} f_K]q  - B(q,h) \in \mathrm{im}\,J$ is thus equivalent to 
\[
v \circ v_{h} - A(\hat{q}^* \circ (\wG (v_{h}+\theta) + \wQ w_{h}))  \in \mM\,.
\]
As $\hat{v} - v \in \mZ$ (\Cref{lemM1M2M3}(e)) and $v_{h} \in \mY$, hence $(\hat{v} - v) \circ v_{h} \in \mY$ (\Cref{lemM1M2M3}(c)), this is equivalent to
\[
\Psi':= \hat{v} \circ v_{h} - A(\hat{q}^* \circ (\wG (v_{h}+\theta) + \wQ w_{h}))  \in \mM\,.
\]
It remains to show that $\Psi \in \mM$ if and only if $\Psi' \in \mM$. The condition $\wJ\hat{h} = \mD g$ translates to:
\begin{equation}
\label{eqtrans1}
\Gamma D_{c_0}[\hat{v}_{\hat{h}}-\theta]=0\,.
\end{equation}
The condition $Jh=\mD f_K$ translates to $v_{h}+\theta \in \mathrm{im}\,[A\,|\,\bm{1}_\mP]$ and
\begin{equation}
\label{eqtrans2}
ZD_{x_0}[\wG (v_{h} + \theta) + \wQ w_{h}]=0\,.
\end{equation}
Applying \Cref{lemtransfer} to Equations~\eqref{eqtrans1}~and~\eqref{eqtrans2} gives $\wG (v_h+\theta) + \wQ w_{h} = \hat{h}^* + t \hat{q}^*$ and $v_{h} = -\theta + \hat{v}_{\hat{h}} + t\hat{v} + z$ for some $t \in \R$ and some $z \in \mZ$. Substituting into the expression for $\Psi'$ gives:
\[
\Psi':= -\hat{v} \circ \theta + \hat{v} \circ \hat{v}_{\hat{h}} + \hat{v} \circ z - A(\hat{q}^* \circ \hat{h}^*) + t(\hat{v}^2-A(\hat{q}^*)^2)\,.
\]
Observing that $\hat{v} \circ z \in \mM$ (\Cref{lemM1M2M3}(c)) and $\hat{v}^2 - A(\hat{q}^*)^2 \in \mM$ (\Cref{lemB2}), we see $\Psi' + \Psi \in \mM$, hence $\Psi' \in \mM$ is equivalent to $\Psi \in \mM$. 
\end{proof}

\section{Some corollaries and conclusions}
\label{seccorollaries}

\subsection{Full-rank networks}

We start with some claims about full-rank networks (i.e., networks such that $n=r$) which are easily inferred from \cite{banajifeliu26}, but can also be derived from the main result of this paper. First observe that in full-rank networks which are not $\mP$-overdetermined, the alternative system is precisely the solvability system $h(\alpha)^W - \k^W=0$ which consists of $m-n-p$ equations in the same number of variables (see \Cref{remspecial}). 

\begin{proposition}
\label[proposition]{propfull1}
Consider an $(n,m,n)$ network. 
\begin{itemize}
\item[(i)] If $m\leq n+p$ then the network forbids fold bifurcations. 
\item[(ii)] Suppose that $m > n+p$, but the network has no more than $n+p$ distinct sources, and whenever reactions $i$ and $j$ share the same source, then they belong to the same member of the partition $\mP$. Then the network forbids fold bifurcations. 
\item[(iii)] If $m\leq n+p+1$ then the network forbids cusp bifurcations. 
\end{itemize}
\end{proposition}
\begin{proof}
Parts (i) and (ii) are immediate from \cite[Theorem~4.24]{banajifeliu26}, but we sketch the proof in light of the results here. An $(n,m,n)$ network with $n+p$ or fewer distinct sources is either $\mP$-overdetermined (in which case the result is trivial), or has exactly $n+p$ distinct sources (hence $m \geq n+p$). In the case $m=n+p$, the alternative system is empty and (i) is trivial (see the final case in \Cref{remspecial}). In (ii), the assumptions about the partition imply that the alternative system is, after clearing denominators, linear (see \cite[proof of Theorem~4.24]{banajifeliu26} for the details), and the result then follows immediately from \Cref{foldcuspmain}. (We remark that clearing denominators is justified by \Cref{prop:prefactor}.)

(iii) From part (i), this is certainly true if $m \leq n+p$, so assume $m=n+p+1$. If the network is $\mP$-overdetermined the result is immediate. Assume otherwise, so that $s_\mP = n$. Consequently, $\mathrm{rank}\,[A\,|\,\bm{1}_\mP] = n+p$ implying that $\wQ$ is empty and $W$ has a single row. The alternative equations consist of a single equation of the form $h(\alpha)^W - \k^W = 0$ in the scalar variable $\alpha$, dependent on the single parameter $\k^W$, and thus a cusp bifurcation is forbidden by \Cref{foldcuspmain}. 
\end{proof}

We remark that although fold bifurcations are forbidden in $(n,n+1,n)$ networks by \Cref{propfull1}(i), other bifurcations, such as Andronov--Hopf, and even Bautin, bifurcations can occur in bimolecular $(n,n+1,n)$ networks \cite{banajiborosnonlinearity}. Although cusp bifurcations are forbidden in $(n,n+2,n)$ networks by \Cref{propfull1}(iii), quadratic $(2,4,2)$ networks can admit three, positive, nondegenerate equilibria (see \cite[Remark~38]{BBH2024smallbif}).

\subsection{$\mP$-toric networks}
Recall that $\mP$-toric networks are defined as networks such that $s_{\mP} = r$. All full-rank networks which admit nondegenerate equilibria are $\mP$-toric; but networks with $r < n$ can also be $\mP$-toric. The special structure of the alternative equations for $\mP$-toric networks can be helpful when computing conditions for fold and cusp bifurcations. In particular, $\mP$-toric networks give rise to {\em block triangular} alternative systems of the form
\[
f_1(\alpha, \k) = 0, \quad f_2(\alpha, \mu, \k, K) = 0\,,
\]
where the solvability system, $f_1(\alpha, \k):=h(\alpha)^W-\k^W=0$, is a system of $m-r-p$ equations in the $m-r-p$ variables $\alpha$, and $f_2(\alpha, \mu, \k,K):=Z[h(\alpha)^{\wG}\circ \mu^{\wQ}]-K = 0$ is a system of $n-s_\mP$ equations in the variables $\alpha$ and the $n-s_\mP$ variables $\mu$. In the special case that $n=r$ (full rank networks), the second subsystem is empty; this case was discussed in the previous section and so we now assume that $r<n$.

Suppose that $(x_0, \k_0)$ satisfies $g(x_0, \k_0)$, equivalently $(\alpha_0, \mu_0, \k_0,K)$ is a zero of $f:=(f_1, f_2)$ where, as usual, $\alpha_0, \mu_0$ and $K$ are defined uniquely by $x_0 = F(\alpha_0, \mu_0, \k_0)$, and $K = Zx_0$ (see \Cref{thm:degensolvability}(i)). Suppose further than $\ker J^{(g)} (x_0, \k_0) \cap \mathrm{im}\,\Gamma$ is one-dimensional, equivalently by \Cref{thm:degensolvability}(ii), $\ker J^{(f)}(\alpha_0, \mu_0, \k_0,K)$ is one-dimensional. Then, by \Cref{foldcuspmain}, $\tilde{g}_K$ satisfies condition B2--B5 at $(x_0, \k_0)$ if and only if $(f_1, f_2)$ satisfies the corresponding condition at $(\alpha_0, \mu_0, \k_0,K)$. In this case, from the block-triangular structure of $\partial_{(\alpha,\mu)}(f_1, f_2)$, one of the following three situations must occur. 
\begin{itemize}
\item[(i)] $\partial_\alpha f_1(\alpha_0, \k_0)$ is singular (with one-dimensional kernel), while $\partial_{\mu} f_2(\alpha_0, \mu_0, \k_0)$ is nonsingular. In this case, applying \Cref{prop:elimination} to the alternative equations, it suffices to examine $f_1$ at $(\alpha_0, \k_0)$ in computations checking conditions B2--B5 for fold and cusp bifurcations.
\item[(ii)] $\partial_{\mu} f_2(\alpha_0, \mu_0, \k_0)$ is singular (with one-dimensional kernel), while $\partial_\alpha f_1(\alpha_0, \k_0)$ is nonsingular. In this case, we may again apply \Cref{prop:elimination}: we use the implicit function theorem to solve locally to get $\alpha=\phi(\k)$ s.t. $\alpha_0 = \phi(\k_0)$ and $f_1(\phi(\k),\k)=0$ for $\k$ in some neighbourhood of $\k_0$, define $\hat{f}(\mu, \k, K) = f_2(\phi(\k), \mu, \k, K)$, and it suffices to examine $\hat{f}$ in computations checking conditions B2--B5 for fold or cusp bifurcation conditions. 
\item[(iii)] Both $\partial_\alpha f_1(\alpha_0, \k_0)$ and $\partial_{\mu} f_2(\alpha_0, \mu_0, \k_0)$ are singular (we still have that $\ker J^{(f)}(\alpha_0, \mu_0, \k_0,K)$ is one-dimensional).
\end{itemize}
In case (i) we end up examining a system of $m-r-p$ equations in as many variables; while in case (ii) we end up examining a system of $n-s_{\mP}$ equations in as many variables. In case (iii) we still have to examine the full system, but we remark that in both cases (ii) and (iii) $\ker J^{(f)}(\alpha_0, \mu_0, \k_0,K)$ is spanned by a vector of the form $(0,q_2)$ where $q_2 \in \ker \partial_{\mu} f_2(\alpha_0, \mu_0, \k_0)$. Consequently, $v\equiv 0$, $\hat{v} \in \mZ$, and $\hat{q}^* = \wQ w$, simplifying many of the expressions which appear in \Cref{lemB2,lemB3,lemB4,lemB5}. 

As an example of how we may apply these observations, we present some consequences for $\mP$-toric, quadratic networks with a single linear conservation law. As usual, let $(x_0, \k_0)$ and $(\alpha_0, \mu_0, \k_0,K)$ satisfy $g(x_0, \k_0)=0$ and $f_K(\alpha_0, \mu_0, \k_0) = 0$ with $x_0 = F(\alpha_0, \mu_0, \k_0)$ and $K = Zx_0$.

\begin{proposition}
\label[proposition]{thmPtoricquad}
Let $\mP$ be the trivial partition. Consider a $\mP$-toric, quadratic $(n,m,n-1)$ network, and suppose that $(x_0, \k_0)$ is a simply degenerate point of $\tilde{g}_K$. Then the network has a nondegenerate cusp bifurcation, unfolded by its parameters, at $(x_0, \k_0)$, if and only if $f_1$ satisfies conditions B2, B4 and B5 at $(\alpha_0, \k_0)$. 

\end{proposition}

\begin{proof}
Let $\wJ:= J^{(g)}(x_0, \k_0)$ and $J:= J^{(f)}(\alpha_0, \mu_0, \k_0, K)$. The assumption of $\mP$-toricity implies that $s_\mP = r = n-1$, and thus $Z$ and $\wQ$ have rank $1$, and $\mu$ is a scalar variable. The mass action system has a nondegenerate cusp bifurcation, unfolded by its parameters, at $(x_0, \k_0)$ if and only if $\tilde{g}_K$ satisfies conditions B2, B4 and B5 at $(x_0, \k_0)$, equivalently, by \Cref{foldcuspmain}, $f=(f_1, f_2)$ satisfies these conditions at $(\alpha_0, \mu_0, \k_0, K)$. 

Suppose that $\partial_\mu f_2(\alpha_0, \mu_0, \k_0)$ is singular, namely $\partial_\mu f_2(\alpha_0, \mu_0, \k_0)=0$. We show that condition B4 cannot be satisfied, and thus nondegenerate cusp bifurcations are impossible (this is intuitively plausible as $f_2$ is at-most-quadratic in the variable $\mu$, see \cite[Proposition 5.13]{banajifeliu26}). Define $q := (0, \mu_0)$, and note that by assumption $\{q\}$ spans $\ker J$. Hence $v=0$. From \eqref{eqhatq}, $\hat{q}^* = \wQ$, consequently $A\wQ \in \mA$. From \Cref{lemB2}, B2 is satisfied, namely $B(q,q) \in \mathrm{im}\,J$ if and only if $A(\hat{q}^*)^2 \in \mM$; and from \Cref{lemB4}, with $h \in J^{-1}(B(q,q))$, B4 fails, namely $C(q,q,q) - 3B(q,h) \in \mathrm{im}\,J$ if and only if $2A(\hat{q}^*)^3 + 3A(\hat{q}^*\circ \hat{h}^*) \in \mM$.

A quick calculation reveals that for any vector $u \in \R^n$, the condition $A u^* \in \mM$ is equivalent to $u \in \mathrm{im}\,\Gamma + \ker \wJ$. By \Cref{simplydegen}, if $\mathrm{rank}\,\wJ = r-1\,\, (= n-2)$, then $\mathrm{im}\,\Gamma + \ker \wJ = \R^n$, and this condition is automatically satisfied. In this case, condition B2 is automatically satisfied, and B4 automatically fails, so in fact both fold and cusp bifurcations are ruled out. (Note that we have not used the assumption that the network is quadratic up to this point.) 

So now assume that $\mathrm{rank}\,\wJ = r\,\, (= n-1)$, hence $\ker \wJ = \mathrm{span}\,\{\hat{q}\} \subseteq \mathrm{im}\,\Gamma$. Then condition B2 is satisfied, namely $A(\hat{q}^*)^2 \in \mM$, if and only if $\hat{q}^2/x_0 \in \mathrm{im}\,\Gamma + \ker \wJ = \mathrm{im}\,\Gamma$, equivalently $A(\hat{q}^*)^2\,\, (=A\wQ^2) \in \mA$. Similarly, condition B4 fails if and only if $2A(\hat{q}^*)^3 + 3A(\hat{q}^*\circ \hat{h}^*) \in \mA$. We now assume that condition B2 is satisfied, and use the fact that the network is quadratic. In this case, $\wQ$ can be chosen to be either a $(-1,0,1)$-vector or a $(0,1,2)$-vector, see Proposition~5.13 and the preceding lemmas in \cite{banajifeliu26}. In the first case $\wQ^3 = \wQ$, while in the second $\wQ^3 = 3\wQ^2-2\wQ$; thus in both cases $\wQ^3 \in \mathrm{span}\{\wQ, \wQ^2\}$. Thus, as $A\wQ, A\wQ^2 \in \mA$, we obtain $A(\hat{q}^*)^3 \,\,(= A\wQ^3) \in \mA \subseteq \mM$. Further, $\Gamma D_{c_0}[A \hat{h}^* + A(\hat{q}^*)^2]=0$ (see \eqref{eqcstar1} in the proof of \Cref{lemB4}); equivalently, $\wJ(\hat{h} + \hat{q}^2/x_0) = 0$; equivalently $\hat{h}^* + (\hat{q}^*)^2 = t\hat{q}^*$ for some $t \in \R$ (as $\ker \wJ = \mathrm{span}\,\{\hat{q}\}$). Consequently, $\hat{q}^* \circ \hat{h}^*  = -(\hat{q}^*)^3 + t(\hat{q}^*)^2 \in \mathrm{span}\{\wQ, \wQ^2\}$, hence $A(\hat{q}^* \circ \hat{h}^*) \in \mA \subseteq \mM$, and condition B4 fails. 

We have shown that if $\partial_\mu f_2(\alpha_0, \mu_0, \k_0)$ is singular (namely, zero) then whenever condition B2 holds, condition B4 cannot be satisfied. So, for cusp bifurcations to be possible, we must have situation (i) above with $\partial_\alpha f_1(\alpha_0, \k_0)$ singular and $\partial_{\mu} f_2(\alpha_0, \mu_0, \k_0)$ nonzero; then, by \Cref{prop:elimination}, $f$ satisfies conditions B2, B4 and B5 at $(\alpha_0, \mu_0, \k_0, K)$ if and only if $\partial_\alpha f_1(\alpha_0, \k_0)$ satisfies conditions B2, B4 and B5 at $(\alpha_0, \k_0)$.  
\end{proof}

\begin{remark}[The conclusions of \Cref{thmPtoricquad} can hold more generally]
The assumptions that $\mP$ is the trivial partition, and the network is quadratic, are only relevant to the conclusions of \Cref{thmPtoricquad} via the fact that $\wQ$ can then be chosen to be a vector with no more than two distinct nonzero values in its entries. If this holds for any $\mP$-toric $(n,m,n-1)$ network, then the conclusions of \Cref{thmPtoricquad} hold. 
\end{remark}

From \Cref{thmPtoricquad}, we have the corollary:

\begin{corollary}[Quadratic networks forbidding cusp bifurcations]
\label[corollary]{propnnn1}
The following classes of networks forbid cusp bifurcations:
\begin{enumerate}
\item[(i)] Quadratic $(n,n,n-1)$ networks. 
\item[(ii)] Quadratic $(n,n+1,n-1)$ networks with affinely dependent sources. 
\end{enumerate}
\end{corollary}
\begin{proof}
Let $\mP$ be the trivial partition. In both cases, if the network is $\mP$-overdetermined, hence degenerate, the result is trivial as no fold or cusp bifurcations can occur. So assume that the network is not $\mP$-overdetermined, namely $s_{\mP} \geq r = n-1$. Then in both cases we must have $s_\mP = n-1$, namely the network is $\mP$-toric: in case (i), $s_{\mP} = \mathrm{rank}[A\,|\,\bm{1}_{\mP}]-1 < m=n$ is automatic; in case (ii), $s_{\mP} =\mathrm{rank}[A\,|\,\bm{1}_{\mP}]-1 <m-1 = n$ follows from affine dependence of the sources. 

(i) In this case, the solvability system is empty, and so, by \Cref{thmPtoricquad}, cusp bifurcations are impossible. 

(ii) In this case, the solvability system $h(\alpha)^W - \k^W = 0$ consists of a single equation in one variable, with exactly one parameter $\k^W$, and so (using \Cref{prop:reparam}) cannot satisfy condition B5; by \Cref{thmPtoricquad}, a nondegenerate, fully unfolded, cusp bifurcation is impossible. 
\end{proof}

\subsection{Closing remarks}
The presentation of results here has been geared to a specific application, namely, mass action networks; however, it is clear that, using the same principles of Gale duality, the main construction can be adapted to the study of fold and cusp bifurcations in more general polynomial systems on the positive orthant. Further, the fact that the exponent matrix is an integer matrix derived from a mass action network, and that the cone $\mC$ arises as the positive flux cone of the same network, are not relevant to the main results; these would apply essentially unchanged in a setting such as that of \cite{regensburger:gale}.

In forthcoming work, we will illustrate the power of some of the techniques in this paper by using them to help us completely classify cusp bifurcations in the smallest bimolecular networks where these can occur: $(2,5,2)$ networks, and $(3,4,2)$ networks with affinely independent sources.

\vskip 5mm
{\bf Acknowledgements.} The research here was supported by Research England under the Expanding Excellence in England (E3) funding stream awarded to MARS: Mathematics for AI in Real-world Systems in the School of Mathematical Sciences at Lancaster University. I would also like to thank Bal\'azs Boros, for many useful discussions which led to the formulation of the results in this manuscript. 

\vskip 5mm
{\bf Leiden declaration.} No AI tools were used at any stage during formulation or checking of the results in this manuscript; nor during the writing of this manuscript.

\appendix
\section{Proofs of \Cref{prop:coord,prop:elimination,prop:prefactor,prop:power}}

\subsection{Changing coordinates} 
\label{app:proofs}

\begin{proof}[Sketch proof of \Cref{prop:coord}]
Define $M := \Psi_{\hat{y}}(\hat{p}_0)$, so that $M^{-1} = \Psi^{-1}_y(p_0)$; and $J := f_y(p_0)$, $\wJ := \hat{f}_{\hat{y}}(\hat{p}_0)$. The multilinear functions $B(\cdot, \cdot)$, $\wB(\cdot, \cdot)$, $C(\cdot, \cdot, \cdot)$ and $\wC(\cdot, \cdot, \cdot)$ are defined analogously. Then the following five claims hold:

{\bf Claim 1 (condition B1).} $\wJ = M^{-1} J M$, so (i) $J$ and $\wJ$ have the same spectrum, and (ii) $\wJ \hat{q} = 0$ if and only if $Jq = 0$ where $q=M\hat{q}$.

{\bf Claim 2 (condition B2).} For any $\hat{h}$, $\wB(\hat{q},\hat{h}) =  \Psi^{-1}_{yy}(q, J M\hat{h}) + M^{-1} B(q,M\hat{h}) + M^{-1}J B^{(\Psi)}(\hat{q},\hat{h})$. In particular, $B(q,q) = M\wB(\hat{q}, \hat{q}) -JB^{(\Psi)}(\hat{q},\hat{q})$, and so $B(q,q) \in \mathrm{im}\,J$ if and only if $\wB(\hat{q},\hat{q}) \in \mathrm{im}\wJ$. 

{\bf Claim 3 (condition B3).} For each $i$, $f_{\beta_i} =M\hat{f}_{\beta_i} -J\Psi_{\beta_i}$. Consequently, $\mathrm{im}\,f_\beta \subseteq \mathrm{im}\,J$ if and only if $\mathrm{im}\,\hat{f}_\beta \subseteq \mathrm{im}\,\wJ$. 

Now assume that $B(q,q) \in \mathrm{im}\,J$, equivalently (by Claim 2), $\wB(\hat{q},\hat{q}) \in \mathrm{im}\wJ$. 

{\bf Claim 4 (condition B4).} Given any $\hat{h} \in \wJ^{-1}(\wB(\hat{q},\hat{q}))$, and defining $h:= M\hat{h} - B^{(\Psi)}(\hat{q},\hat{q}) \in J^{-1}(B(q,q))$, we have
\[
C(q,q,q) - 3B(q,h) = M\left(\wC(\hat{q},\hat{q},\hat{q}) - 3 \wB(\hat{q},\hat{h})\right) -J\left(C^{(\Psi)}(\hat{q},\hat{q},\hat{q}) -3B^{(\Psi)}(\hat{q},\hat{h})\right)\,,
\]
Consequently, $C(q,q,q) - 3B(q,h) \in \mathrm{im}\,J$ if and only if $\wC(\hat{q},\hat{q},\hat{q}) - 3 \wB(\hat{q},\hat{h}) \in\mathrm{im}\,\wJ$. 

Define $\mD := K_1\partial_{\beta_i} + K_2\partial_{\beta_j}$ for some $(K_1,K_2) \neq (0,0)$ and some $i \neq j$.

{\bf Claim 5 (condition B5).} $\mD f = M\mD \hat{f}-J\mD \Psi$, and so $\mD f \in \mathrm{im}\,J$ if and only if $\mD \hat{f} \in \mathrm{im}\wJ$. Suppose this is the case. Then, given $\hat{h} \in \wJ^{-1}(\mD \hat{f})$, and $h := M\hat{h} - \mD\Psi$ (so $h  \in J^{-1}(\mD f)$), 
\[
\mD[f_y] q - B(q,h)  = M\left(\mD[\hat{f}_{\hat{y}}]\hat{q} - \wB(\hat{q}, \hat{h})\right) - J\left(\mD[\Psi_{\hat{y}}]\hat{q} - B^{(\Psi)}(\hat{q},\hat{h})\right)\,.
\]
Consequently, $\mD[f_y] q - B(q,h) \in \mathrm{im}\,J$ if and only if $\mD[\hat{f}_{\hat{y}}]\hat{q} - \wB(\hat{q}, \hat{h}) \in\mathrm{im}\wJ$. 

The previous 5 claims imply the result.
\end{proof}

\subsection{Eliminating variables}

\begin{proof}[Sketch proof of \Cref{prop:elimination}]
Observe that $\phi_x = -\hat{g}_y^{-1}\hat{g}_x$, and that
\[
J = \left(\begin{array}{cc}f_x&0\\g_x&g_y\end{array}\right) = \left(\begin{array}{cc}\hat{f}_x + \hat{f}_y\phi_x&0\\-\phi_x&I\end{array}\right) = \left(\begin{array}{cc}I & -\hat{f}_y\hat{g}_y^{-1}\\0 & \hat{g}_y^{-1}\end{array}\right)\left(\begin{array}{cc}\hat{f}_x&\hat{f}_y\\\hat{g}_x&\hat{g}_y \end{array}\right) = M\wJ\,,
\]
where the final equality defines $M$, and all derivatives are evaluated at $p_0$. 

Observe that $M$ is invertible, and note that, $0 \neq q$ spans $\ker J$ if and only if $q$ spans $\ker\wJ$. Fix such a $q = (q_1, q_2)$, and note that $q_2 = \phi_x q_1$. The equivalence of each condition B2--B5 for $\wF$ and $F$, hence of (a) and (b), will now follow from Claims 1 to 4 below. 

{\bf Claim 1 (condition B2).} $B^{(F)}(q,q) = MB^{(\wF)}(q,q)$, hence $B^{(F)}(q,q) \in \mathrm{im}\,J$ if and only if $B^{(\wF)}(q,q) \in \mathrm{im}\,\wJ$. (i) Differentiating $0 \equiv \hat{g}(x,\phi(x,\beta),\beta)$ gives $0 = B^{(\hat{g})}(q,q) + \hat{g}_y B^{(\phi)}(q_1,q_1)$. (ii) Differentiating $f(x,\beta) \equiv \hat{f}(x,\phi(x,\beta),\beta)$ and applying (i) gives $B^{(f)}(q,q) = B^{(\hat{f})}(q,q) - \hat{f}_y\hat{g}_y^{-1}B^{(\hat{g})}(q,q)$. (iii) Differentiating $g(x,y,\beta) \equiv y - \phi(x,\beta)$ and applying (i) gives $B^{(g)}(q,q) = -B^{(\phi)}(q_1,q_1) = \hat{g}_y^{-1} B^{(\hat{g})}(q,q)$. 

{\bf Claim 2 (condition B3).} For each $i$, $F_{\beta_i} = M\wF_{\beta_i}$, and so $\mathrm{im}\,\wF_\beta \subseteq \mathrm{im}\,\wJ$ if and only if $\mathrm{im}\,F_\beta \subseteq \mathrm{im}\,J$. Differentiating $f$ and $g$ gives $f_{\beta_i} = \hat{f}_{\beta_i} + \hat{f}_y\phi_{\beta_i}$ and $g_{\beta_i} = -\phi_{\beta_i}$; while differentiating $0 \equiv \hat{g}(x,\phi(x,\beta),\beta)$ gives $\phi_{\beta_i}=-\hat{g}_y^{-1}\hat{g}_{\beta_i}$, from which the claim follows. 

Now suppose that $B^{(F)}(q,q) \in \mathrm{im}\,J$, equivalently, by Claim 1, $B^{(\wF)}(q,q) \in \mathrm{im}\,\wJ$. Observe that, by Claim 1, $J^{-1}(B^{(F)}(q,q)) = \wJ^{-1}(B^{(\wF)}(q,q))$. Fix $h = (h_1, h_2) \in J^{-1}(B^{(F)}(q,q))$. 

{\bf Claim 3 (condition B4).} $C^{(F)}(q,q,q) - 3B^{(F)}(q, h) = M[C^{(\wF)}(q,q,q) - 3B^{(\wF)}(q, h)]$, consequently, $C^{(F)}(q,q,q) - 3B^{(F)}(q, h) \in \mathrm{im}\,J$ if and only if $C^{(\wF)}(q,q,q) - 3B^{(\wF)}(q, h) \in \mathrm{im}\,\wJ$. This can be proved following the same template as for Claim 1.

{\bf Claim 4 (condition B5).} Fix $i,j \in \{1, \ldots, n\}$ and $(K_1, K_2) \in \R^2\backslash\{0\}$, and define $\mD:= K_1 \partial_{\beta_i}+ K_2 \partial_{\beta_j}$. Then (i) $\mD F  = M\mD \wF $, hence $\mD F  \in \mathrm{im}\,J$ if and only if $\mD \wF  \in \mathrm{im}\,\wJ$; and (ii) if $\mD F  \in \mathrm{im}\,J$, hence $\mD \wF  \in \mathrm{im}\,\wJ$, then choosing any $\tilde{h} \in J^{-1}(\mD F)$, equivalently, from (i), $\tilde{h} \in \wJ^{-1}( \mD \wF$), $\mD[\partial_{(x,y)}F]q - B^{(F)}(q,\tilde{h}) = M[\mD[\partial_{(x,y)}\wF]q - B^{(\wF)}(q,\tilde{h})]$, hence $\mD[\partial_{(x,y)}F]q - B^{(F)}(q,\tilde{h}) \in \mathrm{im}\,J$ if and only if  $\mD[\partial_{(x,y)}\wF]q - B^{(\wF)}(q,\tilde{h}) \in \mathrm{im}\,\wJ$. 

This completes the proof that (a) $ \Leftrightarrow$ (b). 

We now prove the equivalence of each condition B1--B5 for $F$ at $p_0$ and $f$ at $p_0'$, hence of (b) and (c). As noted earlier, $q=(q_1, q_2) \in \ker J$ if and only if $q_1 \in \ker J^{(f)}$ and $q_2 = \phi_x q_1$. 

{\bf Claim 5.} If $v_1 \in \mathrm{im}\,J^{(f)}$ then $(v_1,v_2) \in \mathrm{im}\,J$ for all $v_2$; and if $(v_1, v_2) \in \mathrm{im}\,J^{(F)}$ for some $v_2$, then $v_1 \in \mathrm{im}\,J^{(f)}$. These are both trivial claims. 

{\bf Claim 6 (condition B1).} $p_0$ is a simply degenerate point of $F$ if and only if $p_0'$ is a simply degenerate point of $f$. This is immediate as the eigenvalues of $J$ are precisely those of $J^{(f)}$ with $n_2$ additional eigenvalues $1$. 

{\bf Claim 7 (condition B2).} $B^{(f)}(q_1,q_1) \in \mathrm{im}\,J^{(f)}$ if and only if $B^{(F)}(q,q) \in \mathrm{im}\,J$. This follows from Claim 5 as:
\[
B^{(F)}(q,q) = \left(\begin{array}{c}B^{(f)}(q_1,q_1)\\-B^{(\phi)}(q_1,q_1)\end{array}\right)\,.
\]

{\bf Claim 8 (condition B3).} For each $i$, $f_{\beta_i} \in \mathrm{im}\,J^{(f)}$ if and only if $F_{\beta_i} \in \mathrm{im}\,J$, and so $\mathrm{im}\,f_\beta \subseteq \mathrm{im}\,J^{(f)}$ if and only if $\mathrm{im}\,F_\beta \subseteq \mathrm{im}\,J$. This follows from Claim 5 as the first component of $F_{\beta_i}$ is $f_{\beta_i}$.

Now suppose that $B^{(F)}(q,q) \in \mathrm{im}\,J$, equivalently (by Claim 7), $B^{(f)}(q_1,q_1) \in \mathrm{im}\,J^{(f)}$. Let $h = (h_1, h_2)$ be defined as before, so that  $B^{(F)}(q,q) = Jh$. Recall that this implies that $h_2 = B^{(\phi)}(q_1, q_1) + \phi_x h_1$. 

{\bf Claim 9 (condition B4).} $C^{(f)}(q_1,q_1,q_1) - 3B^{(f)}(q_1,h_1) \in \mathrm{im}\,J^{(f)}$ if and only if $C^{(F)}(q,q,q) - 3B^{(F)}(q,h) \in \mathrm{im}\,J$. This follows from Claim 5, noting that
\[
B^{(F)}(q,h) = \left(\begin{array}{c}B^{(f)}(q_1,h_1)\\-B^{(\phi)}(q_1, h_1)\end{array}\right)\,, \quad C^{(F)}(q,q,q) = \left(\begin{array}{c}C^{(f)}(q_1,q_1,q_1)\\-C^{(\phi)}(q_1,q_1,q_1)\end{array}\right)\,.
\]

{\bf Claim 10 (condition B5).} Fix $i,j \in \{1, \ldots, n\}$ and $(K_1, K_2) \in \R^2\backslash\{0\}$, and define $\mD:= K_1 \partial_{\beta_i}+ K_2 \partial_{\beta_j}$. (i) From Claim 5, $\mD f \in \mathrm{im} J^{(f)}$ if and only if $\mD F \in \mathrm{im}\,J$. (ii) Assume that $\mD f \in \mathrm{im} J^{(f)}$, equivalently, $\mD F \in \mathrm{im}\,J$. Fix some $\tilde{h} = (\tilde{h}_1, \tilde{h}_2) \in J^{-1}(\mD F)$, so that $J^{(f)}\tilde{h}_1 = \mD f$. Then the first component of $\mD[\partial_{(x,y)}F]q - B^{(F)}(q, \tilde{h})$ is precisely $\mD[f_x]q_1 - B^{(f)}(q_1, \tilde{h}_1)$, and hence, by Claim 5, $\mD[\partial_{(x,y)}F]q - B^{(F)}(q, \tilde{h}) \in \mathrm{im}\,J^{(F)}$ if and only if $\mD[f_x]q_1 - B^{(f)}(q_1, \tilde{h}_1) \in \mathrm{im}\,J^{(f)}$.

This completes the proof that (b) $\Leftrightarrow$ (c). 
\end{proof}

\subsection{Applying linear mappings} 

\begin{proof}[Sketch proof of \Cref{prop:prefactor}]
Let $J:=J^{(f)}$ and $\wJ:=J^{(\hat{f})}$, and observe that $\wJ = M_0J$. Let $q \in \R^n$ be some nonzero vector spanning $\ker J$, hence $\ker \wJ$. Elementary applications of the chain and product rules give us the following two claims.

{\bf Claim 1 (condition B2).} $B^{(\hat{f})}(q, q) = M_0 B^{(f)}(q, q)$, and hence $B^{(\hat{f})}(q, q) \in \mathrm{im}\,\wJ$ if and only if $B^{(f)}(q, q) \in \mathrm{im}\,J$. 

{\bf Claim 2 (condition B3).} (i) $\hat{f}_\gamma = 0$; (ii) For each $i$, $\hat{f}_{\beta_i} = M_0 f_{\beta_i}$, and so $\mathrm{im}\,\hat{f}_\beta \subseteq \mathrm{im}\,\wJ$ if and only if $\mathrm{im}\,f_\beta \subseteq \mathrm{im}\,J$. 

Let us now assume that $B^{(f)}(q, q) \in \mathrm{im}\,J$, hence by Claim 1, $B^{(\hat{f})}(q, q) \in \mathrm{im}\,\wJ$. Fix $h \in J^{-1}(B^{(f)}(q,q))$, and observe from Claim 1 that $h \in \wJ^{-1}(B^{(\hat{f})}(q,q))$. 

{\bf Claim 3 (condition B4).} $C^{(\hat{f})}(q,q,q) - 3B^{(\hat{f})}(q,h) =  M_0\left[C^{(f)}(q,q,q) -3B^{(f)}(q, h)\right]$, and thus $C^{(\hat{f})}(q,q,q) - 3B^{(\hat{f})}(q,h) \in \mathrm{im}\,\wJ$ if and only if $C^{(f)}(q,q,q) -3B^{(f)}(q, h) \in \mathrm{im}\,J$. By the chain rule, we have $C^{(\hat{f})}(q,q,q) = M_0C^{(f)}(q,q,q)+ 3 [\nabla_{q} M]B^{(f)}(q,q)$, where $\nabla_{q} M := \left.\frac{\partial}{\partial t} M(y_0 + tq, \beta_0, \gamma_0)\right|_{t=0}$. We calculate $B^{(\hat{f})}(q, h) = M_0B^{(f)}(q, h) + [\nabla_{q}M] B^{(f)}(q,q)$, and the claim follows. 

Given a pair of parameters, say $(\phi, \psi)$, and constants $(K_1,K_2) \neq (0,0)$, define the operator $\mathcal{D}:= K_1\partial_{\phi} + K_2 \partial_\psi$. 

{\bf Claim 4 (condition B5).} If one parameter belongs to $\gamma$, say $\phi = \gamma_i$, then we choose $K_1=1, K_2=0$, leading to $\mathcal{D}\hat{f} = 0 \in \mathrm{im}\,\wJ$, hence $0 = \tilde{h} \in \wJ^{-1}(\mathcal{D} \hat{f})$ and $\mathcal{D}[\hat{f}_y] q - B^{(\hat{f})}(q,\tilde{h}) = 0$ (i.e., $(\phi, \psi)$ cannot unfold a nondegenerate cusp bifurcation of $\hat{f}$). Now suppose that $(\phi, \psi) = (\beta_i, \beta_j)$. Then (i) $\mathcal{D} \hat{f} = M_0\mathcal{D} f$, hence $\mD f \in \mathrm{im}\,J$ if and only if $\mD \hat{f} \in \mathrm{im}\,\wJ$. (ii) Suppose that (i) holds, so $J^{-1}(\mathcal{D} f)= \wJ^{-1}(\mathcal{D} \hat{f})$, and choose $\tilde{h} \in J^{-1}(\mathcal{D} f)$. We check that $\mathcal{D}[\hat{f}_y] q - B^{(\hat{f})}(q,\tilde{h}) = M_0[\mathcal{D}[f_y] q - B^{(f)}(q,\tilde{h})]$, and so $\mathcal{D}[\hat{f}_y] q - B^{(\hat{f})}(q,\tilde{h}) \in\mathrm{im}\,\wJ$ if and only if $\mathcal{D}[f_y] q - B^{(f)}(q,\tilde{h}) \in \mathrm{im}\,J$. 

This completes the proof.
\end{proof}

\subsection{Applying a diffeomorphism to both sides of equations} 

\begin{proof}[Sketch proof of \Cref{prop:power}]
Let $J:=J^{(f)}$ and $\wJ:=J^{(\hat{f})}$, and recall that $\wJ = M_0J$. Let $q \in \R^n$ be some nonzero vector spanning $\ker J$, hence $\ker \wJ$. We have:

{\bf Claim 1 (condition B2).} $B^{(\hat{f})}(q, q) = M_0 B^{(f)}(q, q)$, and hence $B^{(\hat{f})}(q, q) \in \mathrm{im}\,\wJ$ if and only if $B^{(f)}(q, q) \in \mathrm{im}\,J$. 

{\bf Claim 2 (condition B3).} For each $i$, $\hat{f}_{\beta_i} = M_0 f_{\beta_i}$, and so $\mathrm{im}\,\hat{f}_\beta \subseteq \mathrm{im}\,\wJ$ if and only if $\mathrm{im}\,f_\beta \subseteq \mathrm{im}\,J$. 

Let us now assume that $B^{(f)}(q, q) \in \mathrm{im}\,J$, hence $B^{(\hat{f})}(q, q) \in \mathrm{im}\,\wJ$. Fix $h \in J^{-1}(B^{(f)}(q,q))$, and observe from Claim 1 that $h \in \wJ^{-1}(B^{(\hat{f})}(q,q))$. 

{\bf Claim 3 (condition B4).} $C^{(\hat{f})}(q,q,q) =  M_0C^{(f)}(q,q,q)$, $B^{(\hat{f})}(q,h) =  M_0B^{(f)}(q, h)$, hence $C^{(\hat{f})}(q,q,q) - 3B^{(\hat{f})}(q,h) =  M_0\left[C^{(f)}(q,q,q) -3B^{(f)}(q, h)\right]$, and thus $C^{(\hat{f})}(q,q,q) - 3B^{(\hat{f})}(q,h) \in \mathrm{im}\,\wJ$ if and only if $C^{(f)}(q,q,q) -3B^{(f)}(q, h) \in \mathrm{im}\,J$. 

Given $\beta_i, \beta_j$, and $(K_1,K_2) \neq (0,0)$, define the operator $\mathcal{D}:= K_1\partial_{\beta_i} + K_2 \partial_{\beta_j}$. 

{\bf Claim 4 (condition B5).} (i) $\mathcal{D} \hat{f} = M_0\mathcal{D} f$, hence $\mD f \in \mathrm{im}\,J$ if and only if $\mD \hat{f} \in \mathrm{im}\,\wJ$. (ii) Suppose $\mD f \in \mathrm{im}\,J$, equivalently, $\mD \hat{f} \in \mathrm{im}\,\wJ$, and choose $\tilde{h} \in J^{-1}(\mathcal{D} f)$ (hence, $\tilde{h} \in \wJ^{-1}(\mathcal{D} \hat{f})$). We check that $\mathcal{D}[\hat{f}_y] q - B^{(\hat{f})}(q,\tilde{h}) = M_0[\mathcal{D}[f_y] q - B^{(f)}(q,\tilde{h})]$, and so $\mathcal{D}[\hat{f}_y] q - B^{(\hat{f})}(q,\tilde{h}) \in\mathrm{im}\,\wJ$ if and only if $\mathcal{D}[f_y] q - B^{(f)}(q,\tilde{h}) \in \mathrm{im}\,J$. 

This completes the proof.
\end{proof}

\small
\bibliographystyle{plain}

\end{document}